\documentclass[10pt,reqno]{amsart}
\usepackage{amssymb,mathrsfs,graphicx}
\usepackage{ifthen}
\usepackage{colortbl}
\usepackage{cases,appendix}
\definecolor{black}{rgb}{0.0, 0.0, 0.0}
\definecolor{red}{rgb}{1.0, 0.5, 0.5}
\provideboolean{shownotes} % define a boolean variable
\setboolean{shownotes}{true} % defined as ``true'' or ``false''
\newcommand{\margnote}[1]{
\ifthenelse{\boolean{shownotes}}%
{\marginpar{\raggedright\tiny\texttt{#1}}}%
{}%
}
\newcommand{\hole}[1]{
\ifthenelse{\boolean{shownotes}}%
{\begin{center} \fbox{ \rule {.25cm}{0cm} \rule[-.1cm]{0cm}{.4cm}
\parbox{.85\textwidth}{\begin{center} \texttt{#1}\end{center}} \rule
{.25cm}{0cm}}\end{center}} {} }

\title[Propagation of chaos for aggregation equations]{Propagation of chaos for aggregation equations with no-flux boundary conditions and sharp sensing zones}

\author[Choi]{Young-Pil Choi}
\address[Young-Pil Choi]{\newline Department of Mathematics
    \newline Inha University, 402--751, Incheon, Republic of Korea}
\email{ypchoi@inha.ac.kr}

\author[Salem]{Samir Salem}
\address[Samir Salem]{\newline Centre de Math\'ematiques et Informatique (CMI), \newline
    Universit\'e de Provence, Technop\^ole Ch\^ateau-Gombert, Marseille, France}
\email{samir.salem@univ-amu.fr}

\numberwithin{equation}{section}

\newtheorem{theorem}{Theorem}[section]
\newtheorem{lemma}{Lemma}[section]

\newtheorem{proposition}{Proposition}[section]
\newtheorem{remark}{Remark}[section]
\newtheorem{definition}{Definition}[section]

\newcommand{\R}{\mathbb R}

\newcommand{\pp}{\mathcal P}
\newcommand{\e}{\varepsilon}

\newcommand{\V}{\mathcal{V}}
\newcommand{\lal}{\langle}
\newcommand{\ral}{\rangle}
\newcommand{\mc}{\mathcal{C}}

\newcommand{\lt}{\left}
\newcommand{\rt}{\right}
\newcommand{\pa}{\partial}
 
\newcommand{\mb}{\mathbf{1}}

\newcommand{\bq}{\begin{equation}}
\newcommand{\eq}{\end{equation}}
\newcommand{\om}{\Omega}
\newcommand{\frm}{\frac{1}{2m}}

\newcommand{\LL}{\mathcal{L}}

\newcommand{\IP}{\mathbb{P}}
\newcommand{\E}{\mathbb{E}}

\newcommand{\mo}{\mathcal{O}}

\DeclareMathOperator*{\esssup}{ess\,sup}
\begin{document}
%%%%%%%%%%%%%%%%
\allowdisplaybreaks

\date{\today}


\begin{abstract} We consider an interacting $N$-particle system with the vision geometrical constraints and reflected noises, proposed as a model for collective behavior of individuals. We rigorously derive a continuity-type of mean-field equation with discontinuous kernels and the normal reflecting boundary conditions from that stochastic particle system as the number of particles $N$ goes to infinity. More precisely, we provide a quantitative estimate of the convergence in law of the empirical measure associated to the particle system to a probability measure which possesses a density which is a weak solution to the continuity equation. This extends previous results on an interacting particle system with bounded and Lipschitz continuous drift terms and normal reflecting boundary conditions by Sznitman[J. Funct. Anal., 56, (1984), 311--336] to that one with discontinuous kernels.
\end{abstract}

\maketitle \centerline{\date}

\tableofcontents

%%%%%%%%%%%%%%%%%%%%%%%%%%%%%%%%%%%%%%%%%%%%%%%%%%%%%%%%%%%%%%%%%%%%%%%%%%%%%%%%%5
%
%
%                        Section: Introduction
%
%
%%%%%%%%%%%%%%%%%%%%%%%%%%%%%%%%%%%%%%%%%%%%%%%%%%%%%%%%%%%%%%%%%%%%%%%%%%%%%%%%%
\section{Introduction}\label{intro}

Mathematical modelling of collective behaviors, such as flocks of birds, schools of fish, or aggregation of bacteria, etc, has received a bulk of attention because of its possible applications in the field of engineering, biology, industry, and sociology \cite{Aoki, LPLSFD, PGE, TBL, VCBCS}. These models are usually based on incorporating different mechanisms of interactions between individuals, for instance, a short-range repulsion, a long-range attraction, and an alignment in certain spatial regions. We refer the reader to \cite{CCH, CCP, CHL} and the references therein for recent surveys of collective behavior models. In this current work, we consider the continuity-type model for collective behavior in the presence of diffusion. More precisely, we are interested in the propagation of chaos for interacting diffusing particles with reflecting boundary conditions describing collective behavior of individuals with vision geometrical constraints. Let $(\Omega,\mathcal{F},(\mathcal{F}_t)_{t\geq 0},\mathbb{P})$ be a stochastic basis endowed with a filtration $(\mathcal{F}_t)_{t \geq 0}$. On that stochastic basis let $\{ B^i_t\}_{i=1}^N$ be $N$ independent $d$-dimensional Brownian motions and $\mathcal{O}$ be a bounded open convex set of $\R^d$ with Lipschitz boundary $\pa\mo$. In this setting, our main stochastic integral equations(in short, SIEs) are given by
\begin{equation}\label{sys_sde_R}
\left\{ \begin{array}{ll}
\displaystyle X_t^i = X_0^i+\int_0^t V[\mu_s^N](X_s^i)\,ds+\sqrt{2\sigma}B_t^i-K_t^i, \quad i=1,\cdots, N, \quad t > 0,&\\[2mm]
\displaystyle K_t^i=\int_0^t n
(X_s^i)\,d|K^i|_s,  \quad |K^i|_t=\int_0^t \mb_{\pa \mathcal{O}}
(X^i_s)\,d|K^i|_s, 
\end{array} \right.
\end{equation}
where  $\mu_s^N:=\frac{1}{N}\sum_{i=1}^{N}\delta_{X_s^i}$ and $n(x)$ denotes the outward normal to $\pa \mo$ at the point $x \in \pa\mo$. Here $X_t^i$ is the position of $i$-th particle at time $t \geq 0$,  $K_t^i$ is a $\R^d$-bounded total variation process called the reflecting force, and $V[\mu]$ represents the velocity field non-locally computed in terms of the density:
\bq\label{def_V}
V[\mu](x) = \int_{\R^d} \nabla \varphi(x-y)\mb_{K(w(x))}(y-x)\mu(dy) \quad \mbox{for} \quad \mu \in \pp(\R^d),
\eq
where $\mb_{K(w(x))}$ is the indicator function on the set $K(w(x))\subset \R^d$ and $w$ is an orientational field, and $\nabla \varphi$ is a bounded Lipschitz interaction field.

As the total number of individuals gets large, the particle system leads to a macroscopic description based on the evolution of the probability density by means of mean-field limit. The rigorous derivation of the mean-field equation is well studied for sufficiently regular forces \cite{CCR2, dobru}, and it is extended to the equations with non-Lipschitz forces and noises in \cite{BCC} under some uniform moment bounds conditions. For the deterministic particle system with singular kernels, the rigorous derivation of continuum descriptions is studied in \cite{CCH, HJ}. For the system \eqref{sys_sde_R} without noises and reflecting forces, the rigorous derivation of mean-field limit model can be obtained by employing the similar strategy as in \cite{CCHS}, in which the second order collective behavior models with sharp sensitivity regions are considered. 

Solving the SIEs \eqref{sys_sde_R} is known as Skorokhod problem \cite{Sko}. This kind of problem is studied in \cite{Tan} where a convex domain is considered, and then it is extended to a general domain satisfying some admissible conditions in \cite{Li-Si}. Here, the admissibility roughly means that the domain can be approximated by smooth domains in a certain sense, see \cite[p. 521]{Li-Si}. Later, those conditions on the domain are removed in \cite{Sai} by employing the strategy used in \cite{Tan} and approximating the Skorokhod equation. It is worth mentioning that so far Skorokhod problems related to the propagation of chaos are only studied when that system has drift or force terms regular enough, see \cite{Mc}, i.e., those are bounded and Lipschitz continuous, to the best of  knowledge of the authors. %However, the velocity fields given in \eqref{sys_sde_R} is not Lipschitz and even discontinuous. 
Moreover, the rigorous derivation of mean-field limit of stochastic differential equations(in short, SDEs) with reflecting boundary conditions are only studied in \cite{Si}, see \cite{HS} for the propagation of chaos of one dimensional Vlasov-Poisson-Fokker-Planck system. 

The main purpose of this paper is to extend the result in \cite{Si}, where stochastic differential equations with reflecting boundary conditions and bounded Lipschitz velocity fields are taken into account, to the case of discontinuous velocity fields. To be more precise, we will show that the $N$ interacting processes $(X_t^i)_{i=1}^N$ of the system \eqref{sys_sde_R} well approximates as $N \to \infty$ the processes $(Y_t^i)_{i=1}^N$ to the following nonlinear SIEs:
\begin{equation}\label{sys_NLS_R}
\left\{ \begin{array}{ll}
\displaystyle Y_t^i = Y_0^i+\int_0^t V[\rho_s](Y_s^i)\,ds+\sqrt{2\sigma}B_t^i- \widetilde K_t^i, \quad \mathcal{L}(Y_t^i)=\rho_t, \quad i=1, \cdots, N, \quad t \geq 0,&  \\[2mm]
\displaystyle \widetilde K_t^i=\int_0^t n
(Y_s^i)\,d|\widetilde K^i|_s,  \quad |\widetilde K^i|_t=\int_0^t \mb_{\pa\mathcal{O}}
(Y^i_s)\,d|\widetilde K^i|_s, & \\[3mm]
Y_0^i = X_0^i, \quad i = 1,\cdots,N. &
\end{array} \right.
\end{equation} 
By a straightforward application of It\^o's formula, we find that the probability density function $\rho_t$ is determined by a continuity type equation of the form:
\bq\label{1_pde_R}
\left\{ \begin{array}{ll}
\pa_t \rho_t + \nabla \cdot (\rho_t V[\rho_t]) = \sigma \Delta \rho_t, \quad x \in \mathcal{O}, \quad t > 0, &\\[2mm]
\displaystyle V[\rho_t](x) = \int_{\mathcal{O}} \nabla \varphi(x-y)\mb_{K(w(x))}(y-x) \rho_t(dy), &
\end{array} \right.
\eq
with the following initial data and boundary conditions:
\[
\rho_t(x)|_{t=0} =:\rho_0(x), \quad x \in \mathcal{O} \quad \mbox{and} \quad \lt\lal \sigma \nabla \rho- \rho V[\rho], n \rt\ral = 0 \quad \mbox{on} \quad \pa \mathcal{O}.
\]
The equation \eqref{1_pde_R} with the set $K \equiv \R^d$, i.e., without vision geometrical constrains, is of the classical form, usually called the {\it aggregation equation} \cite{CCH, TBL}. On the other hand, due to the presence of vision geometrical constraints in \eqref{1_pde_R} which comes from the cutoff interaction function $\mb_{K}$ in the velocity field $V$, the individuals at position $x$ are only interacting with others inside the region $K(w(x))$. Considering the vision geometrical constraints is quite natural in the modelling of animal and human behavior, and realistic modelling of collective behaviors should deal with that, see \cite{AIR, CCHS, PT}. In order to show the convergence of some probability measure, as well as for stability estimate for nonlinear PDEs, we use the Wasserstein distance which is defined by
\[
\mathcal{W}_p^p(\mu, \nu) := \inf_{\xi \in \Gamma(\mu,\nu)} \lt(\int_{\R^d \times \R^d} |x-y|^p \xi(dx,dy)\rt) = \inf_{(X\sim \mu,Y \sim \nu)} \E[|X-Y|^p],
\]
for $p \geq 1$ and $\mu, \nu \in \pp_p(\R^d)$, where $\Gamma(\mu,\nu)$ is the set of all probability measures on $\R^d\times \R^d$ with first and second marginals $\mu$ and $\nu$, respectively, and $(X,Y)$ are all possible couples of random variables with $\mu$ and $\nu$ as respective laws. Given the types of diffusions we are taking into account in this current work, quadratic Wasserstein distances, for instance $\mathcal{W}_{2n}$ with $n \in \mathbb{N}$, could seem more convenient. Notice that the propagation of chaos for a system of interacting diffusing particles with normal reflection boundary condition is proved in \cite{Si} by making use of Wasserstein distance of order $4$. However it has already been pointed out in \cite[Remark 3.1]{CCHS} by the authors and their collaborators that the strategy used to deal with those discontinuous kernels does not work in Wasserstein distance of order $p$ with $p \in(1,\infty)$. Thus, the use of either $\mathcal{W}_1$ or $\mathcal{W}_\infty$ becomes essential in our framework due to the form of velocity fields. Compared to the work in \cite{CCHS}, in which the sensitivity region is independent of the position, the singularity of the interaction function is somehow stronger due to the position dependency of the set $K$, and this constrains us to use the infinite Wasserstein distance $\mathcal{W}_\infty$ defined as
\[
\mathcal{W}_{\infty}(\mu,\nu):=\inf_{\xi \in \Gamma(\mu,\nu)} \xi\mbox{-} \esssup_{(x,y) \in \R^d\times \R^d}|x-y|=\inf_{(X\sim \mu,Y \sim \nu)} \IP\mbox{-}\esssup|X-Y| .
\]

Here we introduce several notations used throughout the paper. $|\cdot|$ and $\lal \cdot, \cdot \ral$ denote the Euclidean distance and the standard inner product on $\R^d$, respectively. We also use the notation $|\cdot|$ for the Lebesgue measure of some set or the cardinal of finite index sets when there is no confusion. $\pp(\mo)$ and $\pp_p(\mo)$ stand for the sets of all probability measures and probability measures with finite moments of order $p \in [1,\infty)$ on $\mo$, respectively. The notation for a probability measure and its probability density is often abused for notational simplicity. For a function $f(x)$, $\|f\|_{L^p}$ represents the usual $L^p(\mo)$-norm, and $\|f\|_{L^p\cap L^{q}}:=\|f\|_{L^p}+\|f\|_{L^{q}}$. For $p \in [1,\infty]$ and $T > 0$, $L^p(0,T; E)$ is the set of the $L^p$ functions from an interval $(0,T)$ to a Banach space $E$. We also denote by $C$ a generic positive constant. For a set $A \subset \R^d$, $\overline A$ represents the closure of $A$.

The rest of this paper is organized as follows. In the next subsection, we will give precise statements of our main results on the existence of solutions to the SIEs and the partial differential equations(in short, PDEs), and the propagation of chaos under suitable assumptions on the sensitivity regions. In Section \ref{sec:ext_par}, we present a global existence of solutions to the particle systems \eqref{sys_sde_R}. Section \ref{sec:ext_sde} is devoted to provide the existence and uniqueness of solution to the PDE and its associated nonlinear SIEs. In Section \ref{sec:mf}, we give the details of proof for the propagation of chaos for the systems \eqref{sys_sde_R} with the aid of law of large number like estimates. In Appendix \ref{app_gron}, we provide two Gronwall's type inequalities to be used in the proof. Finally, Appendix \ref{app_b} is devoted to study a representation for solutions of the equation \eqref{1_pde_R} giving some relations between uniqueness of solutions to the SDEs and the PDE, which complements the proof of Theorem \ref{thm_SDE} below.

%%%%%%%%%%%%%%%%%%%%%%%%%%%%%%%
%
%   \subsection{Main results}
%
%%%%%%%%%%%%%%%%%%%%%%%%%%%%%%%

\subsection{Main results}
We first introduce several notations for the set valued function $x\in \R^d \mapsto K(x)\subseteq \R^d$.
\begin{definition}Let $K \subset \R^d$ be a non-empty compact set and $\e >0$. We define the $\e$-boundary of $K$ by:
\[
\partial^{\e}K :=\left\{x+y \ | \ x\in \partial K, |y|\leq \e \right\},
\] 
and also the $\e$-enlargement(resp. $\e$-reduction) $K^{\e,+}$ (resp. $K^{\e,-}$) by
\[
K^{\e,+}:= K \cup \partial^{\e}K \quad \mbox{and} \quad K^{\e,-}:=K \setminus \partial^{\e}K.
\]
Note that $\pa^\e K = K^{\e,+} \setminus K^{\e,-}$ and $(\pa^\e K)^{\delta,+} \subset \pa^{\e + \delta}K$ for $\e > 0$ and $\delta > 0$.
\end{definition}
Using those notations for $K$ together with the so called {\it rope argument} used in \cite{Hauray} for the propagation of chaos of Vlasov-Poission system in one dimension, we present a useful estimate for the cut off interaction function. We refer to \cite[Lemma 2.2]{CCHS} for details of the proof.
 \begin{lemma}\label{lem_est2}For $K\subset \mathbb{R}^d$ For $x_1,y_1,x_2,y_2 \in \R^d$, we have
\[
|\mb_K(y_1 - x_1) - \mb_K(y_2 - x_2)| \leq \mb_{\pa^{2|x_1 - x_2|}K}(y_1 - x_1) + \mb_{\pa^{2|y_1 - y_2|}K}(y_1 - x_1).
\]
\end{lemma}
In this paper, we consider that the set valued function $K$ satisfying the following conditions:
\begin{itemize}
\item[${\bf (H1)}$] $K$ is globally compact, i.e., $K(x)$ is compact and there exists a compact set $\mathcal{K}$ such that $K(x)\subseteq \mathcal{K}\,, \forall \,x\in \mathbb{R}^d$.
\item[${\bf (H2)}$] There exists a family of closed sets $x \mapsto \Theta(x)$  and a constant $C$ independent $\e>0$ such that:
\begin{itemize}
\item[(i)] $\partial K(x) \subseteq \Theta(x) $, for all $x \in \R^d$,
\item[(ii)] $\sup_{x\in \R^d}|\Theta(x)^{\e,+}| \le C \e$, for all $\e \in (0,1)$,
\item[(iii)] $K(x) \Delta K(x') \subseteq \Theta(x)^{C |x-x'|,+}$ for $x,x' \in \R^d$,
\item[(iv)] $\Theta(x) \subseteq \Theta(x')^{C |x-x'|,+}$ for $x,x' \in \R^d$.
\end{itemize}
\end{itemize}
The set-valued function $\Theta$ in the above is a kind of generalized boundary of the set $K$. It is introduced in \cite{CCHS}, where the sensitivity set $K$ depends on the velocity variable, in order to give a sense to the time-derivative of the particle trajectories when they cross the boundary of $K$. It is also used to consider the vision cone with varying angles with respect to the speed. For more details, we refer to \cite[Section 2]{CCHS}.

In the next subsection, we provide several examples of sensitivity sets satisfying the above assumptions. 
\begin{remark}\label{rmk_sym}If $w$ is Lipschitz, then it follows from ${\bf (H1)}$-${\bf (H2)}$ that
$$\begin{aligned}
|K(w(x)) \Delta K(w(x'))| & \leq |K(w(x)) \Delta K(w(x'))|\lt(\mb_{C\|w\|_{Lip}|x-x'| < 1} + \mb_{C\|w\|_{Lip}|x-x'| \geq 1}\rt)\cr
&\leq \lt|\Theta(w(x))^{C |w(x)-w(x')|,+} \rt|\mb_{C\|w\|_{Lip}|x-x'| < 1}+2|\mathcal{K}|\mb_{C\|w\|_{Lip}|x-x'| \geq 1}\cr
&\leq C\|w\|_{Lip}|x-x'|,
\end{aligned}$$
for $x,x' \in \R^d$.
\end{remark}

We now present the main results of this paper. First we are concerned with the global existence of weak solutions to the SIEs \eqref{sys_sde_R}. For this, we recall the definition of weak solutions for the SIEs \eqref{sys_sde_R}.
\begin{definition}\cite{Par} A weak solution of the stochastic integral equation \eqref{sys_sde_R} with initial data $\mathcal{X}^N_0:=(X^1_0,\cdots,X^N_0)$ is a couple $(\mathcal{X}^N_t,\mathcal{B}_t^N,\mathcal{K}^N_t )_ {t\geq 0}$,  $(\Omega,\mathcal{F},(\mathcal{F}_t)_{t\geq 0},\mathbb{P})$ such that
	
	\begin{itemize}
		\item[(i)] $(\Omega,\mathcal{F},(\mathcal{F}_t)_{t\geq 0},\mathbb{P})$ is some stochastic basis,
		
		\item[(ii)] $\mathcal{B}_t^N = (B_t^1, \cdots, B_t^N)$ is a $dN$ dimensional $(\mathcal{F}_t)$-Brownian motion under $\mathbb{P}$, $(\mathcal{X}^N_t ,\mathcal{K}^N_t )$ are some $(\mathcal{F}_t)$-adapted processes,

		\item[(iii)] $\mathbb{P}$-almost surely $\mathcal{X}^N_t,\mathcal{B}^N_t,\mathcal{K}^N_t $ satisfy \eqref{sys_sde_R}.
	
	\end{itemize}
\end{definition}

\begin{theorem}\label{thm_SIE} Let $N \geq 2$. For any initial data $\mathcal{X}^N_0 \in \overline\mo^N$ and for any $T > 0$, there exists at least one weak existence for the system \eqref{sys_sde_R} on the time interval $[0,T]$.
\end{theorem}

We next state the theorem on the existence of solutions to the nonlinear SIEs and its associated PDEs.
\begin{theorem}\label{thm_SDE} Let $\rho_0$ be a probability measure on $\mo \subset \R^d$ satisfying $\rho_0 \in  L^\infty(\mo)$ and let $(X_0^i)_{i=1,\cdots, N}$ be $N$ independent variables with the law $\rho_0$. Suppose that the set-valued functions $K$ satisfies $\bf{(H1)}$-$\bf{(H2)}$, $w$ is Lipschitz, and $\nabla\varphi \in W^{1,\infty}(\R^d)$. Then, for some $T > 0$, there exist a unique pathwise solution $(Y_t^i)_{i=1,\cdots, N}$ to the nonlinear SIEs \eqref{sys_NLS_R} and a unique weak solution $\rho \in L^\infty(0,T; (L^1_+ \cap L^\infty)(\mo)) \cap \mc([0,T];\pp_1(\mo))$ to \eqref{1_pde_R} for the initial condition $\rho_0$, which is the law of solution for \eqref{sys_NLS_R} up to time $T >0$. Moreover if $\tilde{\rho}$ is another weak solution to \eqref{1_pde_R} up to time $T>0$ with initial condition $\tilde{\rho}_0 \in\pp(\mo)$ then for any $t\in [0,T]$ we have
\[
\mathcal{W}_\infty(\rho_t,\tilde{\rho}_t)\leq e^{\int_0^t \|\rho_s\|_{L^1\cap L^{\infty}}ds}\,\mathcal{W}_\infty(\rho_0,\tilde{\rho}_0).
\]
\end{theorem}
\begin{remark}\label{rmk:mom} Since $\mo$ is bounded and the mass is conserved, we easily get
\[
\int_\mo |x|^q \rho_t\,dx \leq (\mbox{diam}(\mo))^q \int_\mo \rho_t\,dx = (\mbox{diam}(\mo))^q \int_\mo \rho_0\,dx \quad \mbox{for} \quad q \geq 1, \quad t \geq 0,
\]
i.e., $\rho \in L^\infty(0,T;\pp_q(\mo))$ for any $q \geq 1$.
\end{remark}

\begin{remark}It is worth emphasizing that our strategy is directly applicable to the whole space case when $\sigma = 0$ under additional assumptions on the initial moment bounds, i.e., there is no diffusion, this is why we specify the regularity of solutions even though there is the inclusion between $L^p$ spaces. Note that it is impossible to define the infinite Wasserstein distance between two solutions to \eqref{1_pde_R} in the whole space. 
%It is worth emphasizing that our strategy is not directly applicable to the whole space case if $\sigma>0$, as it is not possible to define the infinite Wasserstein distance between two solutions to \eqref{1_pde_R} in this case. In the case $\sigma=0$ though, our strategy enables to state the above Theorem in the case $\mo=\R^d$, provided that $\rho_0$ and $\tilde{\rho}_0$ are compactly supported. 
\end{remark}

Before stating our result on propagation of chaos, we recall the definition of a chaotic sequence. We refer to \cite{sznitman} for details of the proof of the equivalence relation in the definition below.
\begin{definition} Let $\rho$ be a probability on some polish space $E$ and $\lt((X_i^N)_{i \le N} \rt)_{N \in \mathbb{N}}$  be a sequence of exchangeable random variables. Let us denote by $(u^N)_{N\in \mathbb{N}}$ the sequence of their laws. Then  $\lt((X_i^N)_{i \le N} \rt)_{N \in \mathbb{N}}$  is $\rho$-chaotic if for any $k\geq 2$ and test functions $\phi_1,\cdots,\phi_k\in \mc^{\infty}(E)$ it holds
\[
\int_{E^N}\phi_1\otimes\cdots\otimes\phi_k\otimes 1\cdots\otimes 1 \, u^N(dx_1,\cdots,dx_N)\to \prod_{i=1}^k\int_{E} \phi_i(x)\, \rho(dx) \quad \mbox{as} \quad N \to \infty.
\]
Or equivalently if
\[
\mu^N := \frac1N \sum_{i=1}^N \delta_{X^N_i} \overset{\LL}{\longrightarrow} \rho \quad \mbox{as} \quad N \to \infty,
\]
where the convergence is in a law, in the space of probability measure on $E$.
\end{definition}
\begin{remark}
Assume that $\rho$ lies in $\mathcal{P}_p(E)$ endowed with the $\mathcal{W}_p$ distance. Then a sufficient condition for the sequence $\lt((X_i^N)_{i \le N} \rt)_{N \in \mathbb{N}}$ to be $\rho$-chaotic is the following:
		$$\E \lt[ \mathcal{W}_p(\mu^N, \rho)\rt] \to 0 \quad \mbox{as} \quad N \to \infty.$$
	\end{remark}

Then starting from an $\rho_0$-chaotic initial condition on the dynamics of  \eqref{sys_sde_R} (in fact, we assume that the random variables $(X^i_0)_{i=1,\cdots,N}$ are independent identically distributed with the law $\rho_0$ which is stronger than being $\rho_0$-chaotic), we show that this chaotic character is preserved on time, more precisely, that solutions at time $t>0$ to  \eqref{sys_sde_R} are $\rho_t$-chaotic where $\rho_t$ is solution at time $t$ to \eqref{1_pde_R}. Moreover, we provide some quantitative estimate in the theorem below.

\begin{theorem}\label{thm:PC_T}  Suppose that the set-valued function $K$ satisfies $\bf{(H1)}$-$\bf{(H2)}$, $w$ is Lipschitz, and $\nabla \varphi \in W^{1,\infty}(\R^d)$. Let $\rho$ be a solution to the equation \eqref{1_pde_R} up to time $T > 0$, such that $\rho \in L^\infty(0,T; (L^1_+ \cap L^\infty)(\mo)) \cap \mc([0,T];\pp_1(\mo))$ with initial data $\rho_0 \in (L^1_+ \cap L^\infty)(\mo) \cap  \pp_1 (\mo)$. Furthermore, assume that $(X_0^i)_{i=1,\cdots, N}$ are $N$ independent variables with the law $\rho_0$.  Then, for any $1 \leq p < q < \infty$ and integer $m$ such that $N\geq (2m)^2$, there exists a constant $C>0$ depending only on $\rho_0$, $q$, $\|w\|_{Lip}$, $\|\nabla_x\varphi\|_{W^{1,\infty}}$, and $T$ such that 
\[
\sup_{t\in[0,T]}\mathbb{E}\lt[\mathcal{W}_p(\mu_t^N,\rho_t)\rt] \leq Cc_mN^{-\frac{1}{2}+\frm}+ \left\{ \begin{array}{ll}
 N^{-1/2p} + N^{-(q-p)/qp} & \textrm{if $2p > d$ and $q \neq 2p$}, \\[2mm]
 N^{-1/2p}\log(1+N)^{1/p} + N^{-(q-p)/qp} & \textrm{if $2p = d$ and $q \neq 2p$},\\[2mm]
 N^{-1/d} + N^{-(q-p)/qp} & \displaystyle \textrm{if $2p < d$ and } q \neq \frac{d}{d-p},
  \end{array} \right.
\]
with $c_m=( 2m!)^\frm\sqrt{8m}+8e^{2m}$ and where $\mu_t^N = \frac1N\sum_{i=1}^N \delta_{X^i_t}$ is the empirical measure associated to the particle system \eqref{sys_sde_R} with initial condition $(X_0^i)_{i=1,\cdots, N}$.
\end{theorem}

%%%%%%%%%%%%%%%%%%%%%%%%%%%%%%%
%
%  \subsection{Examples of sensitivity sets}
%
%%%%%%%%%%%%%%%%%%%%%%%%%%%%%%%
\subsection{Examples of sensitivity sets}\label{ssec:ex} In this part, we list several sensitivity sets satisfying our main assumptions ${\bf (H1)}$-${\bf (H2)}$. It is worth mentioning that, in the majority of cases, we do not need to introduce the generalized boundary set $\Theta$.

{\bf A fixed closed ball.-} If we choose $K = B(0,r) := \{ x \in \R^d : |x| \leq r\}$ with $r > 0$, then it is clear that $B(0,r)$ satisfies the assumptions ${\bf (H1)}$-${\bf (H2)}$ with $\Theta = \pa B(0,r)$. 

{\bf A closed ball with varying radius.-} Let $\bar r: \R_+ \to \R_+$ be bounded and Lipschitz function, and take into account the case $K(x) = B(0,\bar r(|x|))$. In this case, it is easy to check the conditions ${\bf (H1)}$, and ${\bf (H2)}$ with $\Theta(x) = \pa B(0,\bar r(|x|))$ since the symmetric difference is always included in a form of torus which can be expressed by the enlargement of $B(0,\bar r(|x|))$. 

{\bf A vision cone with a fixed angle.-} Let us consider the vision cone with a fixed angle $\theta\in(0,\pi)$ and a radius $r>0$, and a direction $w(x)$, so that the set valued function $K$ is defined as
\[
K(x) = C(r,w(x),\theta):= \lt\{ y\in \R^d : |y|\leq r \, , \cos^{-1}\lt(\frac{\lal y,w(x)\ral}{|y||w(x)|}\rt)\in[-\theta,\theta]  \rt\} \quad \mbox{with} \quad d=2,3.
\]
Suppose that the direction function $w$ is Lipschitz and bounded from the both above and below by some positive constant, i.e., $w^* \geq |w| \geq w_* > 0$. Then it is clear that  $K\circ w$ satisfies ${\bf (H1)}$, and it is not hard to check that satisfies ${\bf (H2)}$ with $\Theta = \pa K$ due to the boundedenss of $w$. For this, similar estimates in \cite{CCHS} can be used. Note that such cutoff interaction function is considered in \cite{CPT} for the dynamics of pedestrians.

{\bf A vision cone with varying angles.-} We now consider the vision cone with varying angles with respect to the speed. For this, we first define the angle function $0 < \theta(z) \in \mc^\infty(\R_+)$ by $\theta(z) = \pi$ for $0 \leq z \leq 1$, $\theta(z)$ is decreasing for $z \geq 1$, and $\theta(z) \to \theta_* > 0$ as $z \to +\infty$. Using this $\theta$ function, we set
\[
K(x) = C(r,w(x),\theta(|w(x)|)):= \lt\{ y\in \R^d : |y|\leq r \, , \cos^{-1}\lt(\frac{\lal y, w(x)\ral}{|y||w(x)|}\rt)\in\lt[-\theta(|w(x)|),\theta(|w(x)|)\rt]  \rt\}, 
\]
with $d=2,3$. Very similar consideration is studied in \cite{CCHS} for second-order collective behavior models. In this case, it is required to use the following generalized boundary set $\Theta$:
\[
\Theta(x):=\begin{cases}
\partial C(r,w(x),\theta(|w(x)|))\cup R(w(x))& \text{ if } |w(x)|\in(1/2,1), \\ 
\partial C(r,w(x),\theta(|w(x)|)) & \text{ else}, 
\end{cases} 
\]
where $R(w(x))=[a(w(x)),b(w(x))]$ with
\[
a(w(x))=-r\frac{w(x)}{|w(x)|} \quad , \quad b(w(x))= 2r(|w(x)|-1)\frac{w(x)}{|w(x)|}.
\]
Then, by assuming the Lipschitz continuity for the direction function $w$ and using similar arguments as in \cite[Section 5.3]{CCHS}, we can check that the above vision cone satisfies the assumptions ${\bf (H1)}$-${\bf (H2)}$.

%%%%%%%%%%%%%%%%%%%%%%%%%%%%%%%
%
%   \subsection{Weak-strong Lipschitz estimate}
%
%%%%%%%%%%%%%%%%%%%%%%%%%%%%%%%
\subsection{Weak-strong Lipschitz estimate}

In order to give a main idea of the proof, we provide a crucial weak-strong Lipschitz estimate for the velocity fields generated by two probability measures under the assumptions ${\bf (H1)}$-${\bf (H2)}$ on the set valued function $K(\cdot)$.
\begin{lemma}\label{lem:roparg}
Let $Y$ and $Y'$ be two random variables on $\mo$ and denote $\rho=\LL(Y)$ and $\rho'=\LL(Y')$. Assume that $\rho\in (L^1 \cap L^{\infty})(\mo)$. Then there exists a constant depending only on $\|\nabla \varphi\|_{W^{1,\infty}},\|w\|_{Lip}$ such that
\bq\label{eq_vec}
|V[\rho](Y)-V[\rho'](Y')|\leq C\|\rho\|_{L^1 \cap L^\infty}\IP\mbox{-}\esssup|Y-Y'|,
\eq
where $V$ is given in \eqref{def_V}.
\end{lemma}
\begin{proof}
Introducing $\pi:=\LL(Y,Y')$, we first decompose the left hand side of \eqref{eq_vec} into three terms:
$$\begin{aligned}
&|V[\rho](Y)-V[\rho'](Y')|\cr
&\quad =\lt|\int\lt( \mb_{K(w(Y))}(y-Y)\nabla\varphi(Y-y)-\mb_{K(w(Y'))}(y'-Y')\nabla\varphi(Y'-y')\rt)\pi(dy,dy')\rt|\cr
&\quad \leq \int\lt( \mb_{K(w(Y))}(y-Y)\nabla\varphi(Y-y)-\mb_{K(w(Y))}(y-Y)\nabla\varphi(Y'-y')\rt)\pi(dy,dy')\cr
&\qquad +\int\lt( \mb_{K(w(Y))}(y-Y)-\mb_{K(w(Y'))}(y-Y)\rt)\nabla\varphi(Y'-y')\pi(dy,dy')\cr
&\qquad +\int\lt( \mb_{K(w(Y'))}(y-Y)-\mb_{K(w(Y'))}(y'-Y')\rt)\nabla\varphi(Y'-y')\pi(dy,dy')\cr
&\quad =: I_1+I_2+I_3. 
\end{aligned}$$
Here $I_i,i=1,2,3$ can be estimated as follows.

\noindent $\diamond$ Estimate $I_1$: Due to the regularity of $\nabla\varphi$, we easily obtain
$$
I_1\leq \|\nabla\varphi\|_{Lip}\lt(\IP\mbox{-}\esssup|Y-Y'|+\pi\mbox{-}\esssup |y-y'|\rt)=2\|\nabla\varphi\|_{Lip}\IP\mbox{-}\esssup|Y-Y'|. 
$$
$\diamond$ Estimate $I_2$: Using our main assumptions together with Remark \ref{rmk_sym} yields
$$\begin{aligned}
		I_2&\leq \|\nabla\varphi\|_{L^\infty}\int \lt|\mb_{K(w(Y))}(y-Y)-\mb_{K(w(Y'))}(y-Y)\rt|\pi(dy,dy')\cr
		&\leq \|\nabla\varphi\|_{L^\infty}\int \mb_{K(w(Y))\Delta K(w(Y'))}(y-Y)\rho(dy)\\
		&\leq C\|\nabla\varphi\|_{L^\infty} \|\rho\|_{L^1\cap L^\infty} |K(w(Y))\Delta K(w(Y'))|\\
		&\leq C\|\nabla\varphi\|_{L^\infty} \|\rho\|_{L^1\cap L^\infty}\|w\|_{Lip}|Y-Y'|\\
		&\leq C\|\nabla\varphi\|_{L^\infty}\|\rho\|_{L^1\cap L^\infty}\IP\mbox{-}\esssup|Y-Y'|.
\end{aligned}$$

$\diamond$ Estimate $I_3$: It follows from Lemma \ref{lem_est2}, and $\bf{(H2)}$ that
$$\begin{aligned}
I_3&\leq \|\nabla\varphi\|_{L^\infty}\int \lt|\mb_{K(w(Y')}(y-Y)-\mb_{K(w(Y'))}(y'-Y')\rt|\pi(dy,dy')\cr
&\leq \|\nabla\varphi\|_{L^\infty}\int \lt(\mb_{\partial^{2|Y-Y'|}K(w(Y'))}(y-Y)+\mb_{\partial^{2|y-y'|}K(w(Y'))}(y-Y)\rt)\pi(dy,dy')\\
&\leq \|\nabla\varphi\|_{L^\infty}\int \lt(\mb_{\partial^{2\IP\mbox{-}\esssup |Y-Y'|}K(w(Y'))}(y-Y)+\mb_{\partial^{2\pi\mbox{-}\esssup |y-y'|}K(w(Y'))}(y-Y)\rt)\rho(dy)\\
&\leq  C\|\nabla\varphi\|_{L^\infty} \|\rho\|_{L^1\cap L^\infty}\lt(\IP\mbox{-}\esssup |Y-Y'|+\pi\mbox{-}\esssup |y-y'|\rt)\\
&\leq C\|\nabla\varphi\|_{L^\infty} \|\rho\|_{L^1\cap L^\infty} \IP\mbox{-}\esssup |Y-Y'|.
\end{aligned}$$
By combining all the above estimates, we conclude our desired result.
\end{proof}
\begin{remark} As mentioned in Introduction, we are imposed to use the infinite Wasserstein distance to have the above weak-strong Lipschitz estimate due to the stronger singularity in the velocity fields than the one in \cite{CCHS}. Note that in \cite[Proposition 2.3]{CCHS} the similar estimate is obtained in the Wasserstein distance of order 1.
\end{remark}
%%%%%%%%%%%%%%%%%%%%%%%%%%%%%%%%%%%%%%%%%%%%%%%%%%%%%%%%%%%%%%%%%%%%%%%%%%%%%%%%%%%%%%%%%%%%%%%%%%%%%%%%%%%%%%%%%%%%%%%%%%%%%%%%%%%%%%%%%%%%%%
%
%
%         \section{Existence of solution to particles system}
%
%
%%%%%%%%%%%%%%%%%%%%%%%%%%%%%%%%%%%%%%%%%%%%%%%%%%%%%%%%%%%%%%%%%%%%%%%%%%%%%%%%%%%%%%%%%%%%%%%%%%%%%%%%%%%%%%%%%%%%%%%%%%%%%%%%%%%%%%%%%%%%%%%%

\section{Global existence of weak solutions for the SIEs}\label{sec:ext_par}
In this section, we provide the details of proof of Theorem \ref{thm_SIE} on the global existence of weak solutions to the stochastic particle system \eqref{sys_sde_R}. The proof relies on an adapted use of Girsanov's Theorem which is useful for SDEs with non-smooth drift but additive noise. 

\subsection{Proof of Theorem \ref{thm_SIE}} Let $(\Omega,\mathcal{F},(\mathcal{F}_t)_{t\geq 0},\mathbb{P})$ be a stochastic basis and $(\mathcal{B}_t^N)_{t \geq 0}$ be a $dN$ dimensional $(\mathcal{F}_t)$-Brownian motion on this basis. We define a $\R^{dN}$-valued function $\mathcal{A}^N := (\mathcal{A}^N_1, \cdots, \mathcal{A}^N_N)$ by 
\[
\mathcal{A}_i^N(x_1,\cdots,x_N) := V[\mu^N](x_i),
\]
where $\mu^N:=\frac{1}{N}\sum_{j=1}^N\delta_{x_j}$. Next we define
\begin{align*}%\label{ext_sdeR}
\begin{aligned}
\mathcal{X}_t^N &:=\mathcal{X}^N_0+\sqrt{2\sigma}\mathcal{B}_t^N -\mathcal{K}^N_t, \quad \mathcal{K}^N_t:=(K_t^1,\cdots,K_t^N), \quad t > 0,\\
K_t^i &:=\int_0^t n (X_s^i)\,d|K^i|_s \ , \ |K^i|_t :=\int_0^t\mb_{\pa\mo}(X_s^i)\,d|K^i|_s, \quad \mbox{for } i = 1,\cdots,N,
\end{aligned}
\end{align*}
for which the well-posedness, together with the fact $\mathcal{X}_t^N\in \mo^N$ $\mathbb{P}$-a.s. for all $t\geq0$,  is ensured by \cite[Theorem 3.1]{Li-Si} or \cite{Tan}. Then we define
\bq\label{ext_sdeR2}
\mathcal{Y}^N_t :=\mathcal{B}^N_t-\frac{1}{\sqrt{2\sigma}}\int_0^t\mathcal{A}^N\left(\mathcal{X}_s^N \right)ds.
\eq
This implies that $(\mathcal{X}^N_t)_{t \geq 0}$ and $(\mathcal{Y}^N_t)_{t \geq 0}$ satisfy
$$\begin{aligned}				
&\mathcal{X}_t^N=\mathcal{X}^N_0+\int_0^t\mathcal{A}^N\left(\mathcal{X}_s^N \right)ds+\sqrt{2\sigma}\mathcal{Y}^{N}_t-\mathcal{K}^N_t,\\
& K_t^i=\int_0^t n(X_s^i)\,d|K^i|_s \ , \ |K^i|_t=\int_0^t\mb_{\pa\mo}(X_s^i)\,d|K^i|_s, \quad \mbox{for } i = 1,\cdots,N.
\end{aligned}$$
Note that the above stochastic integral equation has the same form with \eqref{sys_sde_R}. We now look for a proper stochastic basis under which $(\mathcal{Y}^N_t)_{t \geq 0}$ is a $dN$ dimensional Brownian motion. 
Fix $T>0$ and define $(Z_t^N)_{t \geq 0}$ by
\[
Z_t^N:=\exp\lt(\int_0^t \lt\lal\mathcal{A}^N\left( \mathcal{X}_s^N \right), dB_s^N \rt\ral-\frac{1}{2}\int_0^t |\mathcal{A}^N\left( \mathcal{X}_s^N \right)|^2\,ds\rt).
\]
Since
\[
\lt|\mathcal{A}^N\left( \mathcal{X}_s^N \right)\rt|\leq \|\nabla\varphi\|_{L^\infty} \quad \mbox{a.s.},
\]
we obtain 
\[
\mathbb{E}\lt[\exp\lt(\frac12\int_0^T \lt|\mathcal{A}^N\left( \mathcal{X}_s^N \right)\rt|^2 \,ds\rt)\rt]\leq e^{\|\nabla\varphi\|_{L^\infty}^2 T/2}.
\]
This together with the classical exponential martingale theory yields that the process $(Z_t^N)_{t \in [0,T]}$ is a positive martingale with respect to $(\mathcal{F}_t)_{t\in[0,T]}$ and $\E[Z_t^N] = 1$ for $t\in [0,T]$. Then, by Girsanov's Theorem(see for instance \cite[Theorem 2.51]{Par}), the stochastic process $\mathcal{Y}_t^N$ defined in \eqref{ext_sdeR2} is a $dN$ dimensional Brownian motion under the probability measure $\mathbb{Q}$ defined by
\[
\mathbb{Q}(A)=\int_A Z_t^Nd\mathbb{P}, \quad \mbox{for}\quad A\in \mathcal{F}_t.
\]
This concludes that the following couple 
	\[
\left ( \lt(\Omega,\mathcal{F}, (\mathcal{F}_t)_{t\in [0,T]},\mathbb{Q} \rt), (\lt(\mathcal{X}_t^N\rt)_{t\in [0,T]},\mathcal{Y}_t^N)_{t\geq 0},\mathcal{K}_t^N)_{t\in [0,T]}\right ),
	\]
	is a weak solution to \eqref{sys_sde_R}.

%%%%%%%%%%%%%%%%%%%%%%%%%%%%%%%%%%%%%%%%%%%%%%%%%%%%%%%%%%%%%%%%%%%%%%%%%%%%%%%%%%%%%%%%%%%%%%%%%%%%%%%%%%%%%%%%%%%%%%%%%%%%%%%%%%%%%%%%%%%%%%
%
%
%         Section: Nonlinear SDE existence and stabitlity
%
%
%%%%%%%%%%%%%%%%%%%%%%%%%%%%%%%%%%%%%%%%%%%%%%%%%%%%%%%%%%%%%%%%%%%%%%%%%%%%%%%%%%%%%%%%%%%%%%%%%%%%%%%%%%%%%%%%%%%%%%%%%%%%%%%%%%%%%%%%%%%%%%%%
\section{Existence and stability of the nonlinear SIEs and PDEs}\label{sec:ext_sde}
In this section we study the existence and uniqueness of solutions to the nonlinear SDEs \eqref{sys_NLS_R} which process solutions have time marginals solutions to the continuity equation \eqref{1_pde_R}. As mentioned in Introduction, the existence of such process solutions are studied in \cite{Tan} where $\mo$ is an open convex, and later it is refined in \cite{Li-Si, Si} for the case where $\mo$ is an open domain satisfying the uniform exterior sphere condition which reads
\bq
\label{eq_extc1}
\exists r_0 > 0 \quad, \forall\,x \in \pa\mo, \quad \exists y_x \in \R^d \quad \mbox{such that} \quad B(y_x,r_0) \cap \overline\mo = \{x\}.
\eq
%\bq
%\label{eq_extc1}
%\exists\, r_0>0, \quad \forall\, x\in \pa \mo, \quad \exists\, y\in \R^d \quad \mbox{such that} \quad  |y-x|=r_0  \quad \mbox{and} \quad B(y,r_0)\cap \mo=\emptyset.
%\eq
We here set ourselves in the case where $\mo$ is convex. Note that it implies that
\bq\label{eq_conv}
n(w)\cdot(w-w') \geq 0 \qquad \mbox{for any } w\in \pa\mo \mbox{ and } w'\in \overline\mo,
\eq
and if for some $w\in \pa \mo$ some vector $k\in \R^d$ satisfies
\[
k\cdot(w-w') \geq 0 \qquad \mbox{for any } w'\in \overline\mo,
\]
then it holds $k=\theta n(w)$ for some $\theta>0$. Later, we will use these observations for the existence of strong solutions to the system \eqref{sys_NLS_R}, see the proof of Proposition \ref{prop:exNLSDER} below.
\subsection{Regularized system}
In this part, we introduce a regularized system, and show the uniform boundedness of solutions to that regularized system in regularization parameters. Consider a mollified interaction function $\mb_{K}^{\e,\eta}$ defined by
\[
\mb^{\e,\eta}_{K(x')}(x) = \int_{\R^d} \mb_{K(x'-y')}(y-x)\phi_\e(y)\phi_\eta(y')\,dydy',
\]
and consistently with the notation introduced before we set
\[
V^{\e,\eta}[\mu](x):=\int_{\mo}\mb^{\e,\eta}_{K(w(x))}(y-x)\nabla\phi(x-y)\mu(dy).
\]
\begin{lemma} 
	Assume $\rho_0\in (L^1_+\cap L^{\infty})(\mo) \cap \pp_1(\mo)$, let  $Y_0$ be with the law $\rho_0$ and $(B_t)_{t\geq 0}$ be a Brownian motion independent of $Y_0$. Then for any $T>0$, there exists a unique process solving the following nonlinear SIEs up to time $T > 0$ in the strong sense:
\begin{equation}\label{ApproNLSDER}
\left\{ \begin{array}{ll}
\displaystyle Y_t^{\e,\eta}=Y_0+\int_0^t\left (\int \mb^{\e,\eta}_{K(w(Y^{\e,\eta}_s))}(y-Y_s^{\eta,\e})\nabla\varphi(Y_s^{\eta,\e}-y)\rho^{\e,\eta}_s(dy)\right)ds+\sqrt{2\sigma}B_t-K^{\e,\eta}_t, & \\[5mm]
\displaystyle K^{\e,\eta}_t=\int_0^t  n (Y_s^{\e,\eta})\,d|K^{\e,\eta}|_s, \quad |K^{\e,\eta}|_t=\int_0^t \mb_{\pa\mo}(Y_s^{\e,\eta})\,d|K^{\e,\eta}|_s, & \\[5mm]
\mathcal{L}(Y_t^{\e,\eta})= \rho^{\e,\eta}_t . &
\end{array} \right.
\end{equation}
\end{lemma}
\begin{proof}
Since both drift and diffusion terms in the above regularized SIEs are smooth, we deduce from \cite{Si} the strong existence and uniqueness of the process $(Y_t^{\e,\eta})$ to the system \eqref{ApproNLSDER}. %Moreover it is clear that $\rho_t^{\e}$ has a compact support in velocity due to the fact that $Y_t^{\eta,\e}\in \overline\mo$  $\mathbb{P}$-a.s.
\end{proof}

We next provide some basic properties of the mollified indicator function in the lemma below. The proof of that can be found in \cite[Lemma 4.2]{CCHS}.

\begin{lemma} Denote
\[
\mb^{\e}_{K(x')}(x) = \int_{\R^d} \mb_{K(x')}(y-x)\phi_\e(y)\,dy.
\]
(i) For all $ \e > 0$, it holds
\bq\label{lem_diff1}
\int \lt|\mb_{K}^{\e}(x) - \mb_{K}(x)\rt| dx \leq |\pa^{2\e}K|.
\eq
(ii) For all $\e > 0$ and $x_1,y_1,x_2,y_2,v \in \R^d$ we have
\bq\label{lem_bdy}
|\mb_{K}^{\e}(y_1 - x_1) - \mb_{K}^{\e}(y_2 - x_2)| \leq \mb^{\e}_{\pa^{2|x_1 - x_2|}K}(y_1 - x_1) + \mb^{\e}_{\pa^{2|y_1 - y_2|}K}(y_1 - x_1).
\eq
(iii) For all $x \in \mo$ and $0 < \eta \leq 1$, it holds
\bq\label{lem_add}
\int_{\R^d}\lt|\mb_{K(w(x))}^{\eta,\e}(y-x) - \mb_{K(w(x))}^{\e}(y-x)\rt| dy \leq C\eta,
\eq
where $C$ is a positive constant independent of $\e$ and $\eta$.\newline
\end{lemma}

In the proposition below, we show the existence of weak solutions to the corresponding continuity equation to \eqref{ApproNLSDER} with no-flux boundary condition. We also provide a uniform bound estimate of the solution in regularization parameters.
\begin{proposition}\label{lem:EstLinf}
The family of time marginals $(\rho_t^{\e,\eta})_{t\geq 0}$ of the solution to \eqref{ApproNLSDER}, is a global-in-time weak solution of 
\begin{equation}\label{ApproNLPDER}
\left\{\begin{matrix}
\displaystyle \partial_t \rho_t^{\e,\eta}  +\nabla \cdot \left(V^{\e,\eta}[\rho_t^{\e,\eta}]\rho_t^{\e,\eta}\right)=\sigma\Delta \rho_t^{\e,\eta}, \quad x \in \mo, \quad t > 0, \\[2mm] 
\displaystyle  V^{\e,\eta}[\rho_t^{\e,\eta}](x)=\int_{\mo} \,\mb^{\e,\eta}_{K(w(x))}(y-x)\nabla\varphi(x-y)\rho^{\e,\eta}_t(dy),  \\[2mm]
\displaystyle \lt\lal\sigma\nabla \rho_t^{\e,\eta}-\rho_t^{\e,\eta}V^{\e,\eta}[\rho_t^{\e,\eta}], n \rt\ral=0 \quad \mbox{on} \quad \pa\mo,
\end{matrix}\right.
\end{equation}
with $\rho_0=\mathcal{L}(Y_0)\in (L^1\cap L^{\infty})(\mo)$.  Furthermore, there exist a time $T>0$ and a constant $C>0$ such that 
\[
\sup_{0 \leq t \leq T}\| \rho_t^{\e,\eta}\|_{L^1\cap L^\infty}\leq C,
\]
where $C > 0$ is independent of $\e,\eta > 0$.
\end{proposition}

\begin{proof} $\bullet$ ({\it Existence of weak solutions}): For $\phi \in \mc^\infty(\mo)$ satisfying $\lal \nabla \phi(x), n(x)\ral=0$ for $x \in \pa\mo$, by applying It\^o's formula, we get
$$\begin{aligned}
\phi(Y_t^{\e,\eta})&=\phi(Y_0)+\int_0^t \left \langle  \nabla \phi(Y_s^{\e,\eta}),dY_s^{\e,\eta} \right \rangle+\frac{1}{2}\int_0^t (d Y_s^{\e,\eta})^t \,\nabla^2 \phi(Y_s^{\e,\eta})\, d Y_s^{\e,\eta} \\
&=\phi(Y_0)+\int_0^t \lal V^{\e,\eta}[\rho_s^{\e,\eta}](Y_s^{\e,\eta}), \nabla \phi(Y_s^{\e,\eta})\ral ds +\sqrt{2\sigma}\int_0^t \left \langle \nabla\phi(Y_s^{\e,\eta}),dB_s \right \rangle\\
&\quad  -\int_0^t \left \langle \nabla \phi(Y_s^{\e,\eta}),dK^{\eta,\e}_s \right \rangle+\sigma \int_0^t \Delta \phi(Y_s^{\e,\eta})\,ds.
\end{aligned}$$
Since
\[
\int_0^t \left \langle \nabla \phi(Y_s^{\e,\eta}),dK^{\e,\eta}_s \right \rangle ds=\int_0^t \left \langle \nabla \phi(Y_s^{\e,\eta}), n (Y_s^{\e,\eta}) \right \rangle d|K^{\e,\eta}|_s=0,
\]
by taking the expectation and using the fact that $\rho^{\e,\eta}_t =\mathcal{L}(Y^{\e,\eta}_t)$, we obtain
$$\begin{aligned}
\int_{\mo} \phi(x)\rho_t^{\e,\eta}(dx) &=\int_{\mo} \phi(x)\rho_0(dx)+\int_0^t\int_{\mo} V^{\e,\eta}[\rho_s^{\e,\eta}](x)\cdot\nabla \phi(x)\,\rho^{\e,\eta}_s(dx)\,ds +\sigma\int_0^t \int_{\mo} \Delta\phi(x)\rho^{\e,\eta}_s(dx)\,ds.
\end{aligned}$$
This implies that the family of time marginals of the process solutions to \eqref{ApproNLSDER} is a weak solution for the equation \eqref{ApproNLPDER}. \newline
	
$\bullet$ ({\it Uniform bound estimate}): It is straightforward to get that for $p \geq 1$	
$$\begin{aligned}
\frac{d}{dt} \int_{\mo} (\rho_t^{\e, \eta})^p\,dx	&= p\int_{\mo} \partial_t \rho^{\e, \eta}_t(\rho_t^{\e, \eta})^{p-1}dx\\
&\quad =- p\int_{\mo} \nabla \cdot (V^{\e, \eta}[\rho_t^{\e, \eta}]\rho_t^{\e, \eta})(\rho_t^{\e, \eta})^{p-1}dx+\sigma p\int_{\mo}\Delta \rho_t^{\e, \eta}(\rho_t^{\e, \eta})^{p-1}dx\\
&\quad =: I_1 + I_2,
\end{aligned}$$
where $I_i,i=1,2$ are estimated as follows.
$$\begin{aligned}
I_1 &= p\int_\mo V^{\e, \eta}[\rho^{\e, \eta}_t]\rho^{\e, \eta}_t \cdot \nabla \lt((\rho^{\e, \eta}_t)^{p-1} \rt) - p\int_{\pa\mo}V^{\e, \eta}[\rho^{\e, \eta}_t]\rho^{\e, \eta}_t \cdot n(x) (\rho^{\e, \eta}_t)^{p-1}\,dS(x) \cr
&= - (p-1) \int_{\mo} \lt(\nabla \cdot V^{\e, \eta}[\rho_t^{\e, \eta}] \rt)(\rho_t^{\e, \eta})^p\,dx -  p\int_{\pa\mo}V^\e[\rho^{\e, \eta}_t]\rho^{\e, \eta}_t \cdot n(x) (\rho^{\e, \eta}_t)^{p-1}\,dS(x) ,\cr
I_2 &= - \sigma p \int_{\mo}  \nabla \rho_t^{\e, \eta}\cdot\nabla \lt((\rho_t^{\e, \eta})^{p-1}\rt)dx + \sigma p \int_{\pa\mo} \nabla \rho_t^{\e, \eta}\cdot n(x)(\rho_t^{\e, \eta})^{p-1}\,dS(x)\cr
&= - \sigma p(p-1)\int_{\mo} (\rho_t^{\e, \eta})^{p-2} |\nabla \rho_t^{\e, \eta}|^2 dx + \sigma p \int_{\pa\mo} \nabla \rho_t^{\e, \eta}\cdot n(x)(\rho_t^{\e, \eta})^{p-1}\,dS(x).\cr
\end{aligned}$$
Thus we have
$$\begin{aligned}
\frac{d}{dt}\|\rho^{\e,\eta}_t\|_{L^p}^p &= -(p-1) \int_{\mo} \lt(\nabla \cdot V^{\e,\eta}[\rho_t^{\e,\eta}] \rt)(\rho_t^{\e,\eta})^p\,dx- \sigma p(p-1)\int_{\mo} (\rho_t^{\e,\eta})^{p-2} |\nabla \rho_t^{\e,\eta}|^2 dx\cr
&\leq (p-1)\|\nabla \cdot V^{\e,\eta}[\rho_t^{\e,\eta}]\|_{L^\infty}\|\rho^{\e,\eta}_t\|_{L^p}^p,
\end{aligned}$$
due to the boundary condition. On the other hand, we can estimate
\bq\label{est_div}
\|\nabla \cdot V^{\e,\eta}[\rho_t^{\e,\eta}]\|_{L^\infty} \leq C\|\rho^{\e,\eta}_t\|_{L^\infty},
\eq
where $C > 0$ is independent of $\e, \eta$, and $p$. Indeed, for $i=1,\cdots,d$, and $|h| \leq 1$, we get
$$\begin{aligned}
&\int_\mo \lt(\pa_i \varphi(x + he_i -y)\mb^{\e,\eta}_{K(w(x + he_i))}(y-x - he_i) - \pa_i \varphi(x -y)\mb^{\e,\eta}_{K(w(x))}(y-x)\rt)\rho^{\e,\eta}_t(y)\,dy\cr
&\quad = \int_\mo \lt(\pa_i \varphi(x + he_i -y)- \pa_i \varphi(x -y)\rt)\mb^{\e,\eta}_{K(w(x + he_i))}( y - x - he_i) \rho^{\e,\eta}_t(y)\,dy\cr
&\qquad + \int_\mo \pa_i \varphi(x -y)\lt(\mb^{\e,\eta}_{K(w(x + he_i))}(y - x - he_i) - \mb^{\e,\eta}_{K(w(x))}(y - x)\rt)\rho^{\e,\eta}_t(y)\,dy,
\end{aligned}$$
where the first term on the right hand side of the above equality can be easily estimated as
\[
\lt|\int_\mo \lt(\pa_i \varphi(x + he_i -y)- \pa_i \varphi(x -y)\rt)\mb^{\e,\eta}_{K(w(x + he_i))}(y - x - he_i) \rho^{\e,\eta}_t(y)\,dy\rt| \leq h\|\nabla \varphi\|_{Lip}\|\rho^{\e,\eta}_t\|_{L^1}.
\]
For the estimate of second term, we use Lemma \ref{lem_est2}, Remark \ref{rmk_sym}, \eqref{lem_bdy}, together with the assumption ${\bf (H2)}$ to find
$$\begin{aligned}
&\lt|\int_\mo \pa_i \varphi(x -y)\lt(\mb^{\e,\eta}_{K(w(x + he_i))}(y - x - he_i) - \mb^{\e,\eta}_{K(w(x))}(y - x)\rt)\rho^{\e,\eta}_t(y)\,dy\rt|\cr
&\quad = \bigg|\int_{\mo \times \R^d} \pa_i \varphi(x -y)\lt(\mb^{\e}_{K(w(x + he_i)-y')}(y - x - he_i) - \mb^{\e}_{K(w(x+h e_i)-y')}(y - x)\rt)\rho^{\e,\eta}_t(y)\phi_\eta(y')\,dydy'\cr
&\qquad + \int_{\mo \times \R^d} \pa_i \varphi(x -y)\lt(\mb^{\e}_{K(w(x + he_i)-y')}(y - x) - \mb^{\e}_{K(w(x)-y')}(y - x)\rt)\rho^{\e,\eta}_t(y)\phi_\eta(y')\,dydy'\bigg|\cr 
&\quad \leq \|\nabla \varphi\|_{L^\infty}\int_{\mo \times \R^d} \mb^\e_{\pa^{2h}K(w(x+ h e_i)-y')}(y- x - he_i)\rho^{\e,\eta}_t(y)\phi_\eta(y')\,dydy' \cr
&\qquad + \|\nabla \varphi\|_{L^\infty}\int_{\mo \times \R^d} \mb^\e_{K(w(x+h e_i)-y') \Delta K(w(x)-y')}(y-x)\rho^{\e,\eta}_t(y)\phi_\eta(y')\,dydy'\cr
&\quad \leq \|\nabla \varphi\|_{L^\infty}\|\rho^{\e,\eta}_t\|_{L^\infty}\lt(\sup_{x \in \R^d}|\pa^{2h}K(x)| + \int_{\R^d} |K(w(x+h e_i)-y') \Delta K(w(x)-y')|\phi_\eta(y')\,dy' \rt)\cr
&\quad \leq C h \|\nabla \varphi\|_{L^\infty}\|\rho^{\e,\eta}_t\|_{L^\infty}\lt(1 + \|w\|_{Lip} \rt).
\end{aligned}$$
This proves that inequality \eqref{est_div} holds. Thus we obtain
\[
\frac{d}{dt}\left\|\rho_t^{\e, \eta} \right\|^p_{L^p} \leq Cp\|\rho_t^{\e, \eta}\|_{L^\infty}\left\|\rho_t^{\e, \eta} \right\|^p_{L^p}.
\]
By applying Gronwall's inequality, we get
\[
\|\rho_t^{\e, \eta}\|_{L^p}^p \leq \|\rho_0\|_{L^p}^p \exp\lt(Cp\int_0^t \|\rho_s^{\e, \eta}\|_{L^\infty}ds \rt)\quad  \mbox{i.e.,} \quad 
\|\rho_t^{\e, \eta}\|_{L^p} \leq \|\rho_0\|_{L^p} \exp\lt( C\int_0^t \|\rho_s^{\e, \eta}\|_{L^\infty}ds \rt).
\]
We then send $p \to \infty$ to find
\[
\|\rho_t^{\e, \eta}\|_{L^\infty} \leq \|\rho_0\|_{L^\infty} \exp\lt( C\int_0^t \|\rho_s^{\e, \eta}\|_{L^\infty}ds \rt).
\]
We finally use Lemma \ref{lem_gron} (i) with $f(t) = \|\rho_t^{\e, \eta}\|_{L^\infty}$ to have
\[
\|\rho_t^{\e, \eta}\|_{L^\infty} \leq \frac{\|\rho_0\|_{L^\infty}}{1 - C\|\rho_0\|_{L^\infty}t},
\]
which concludes the proof.
\end{proof}
\subsection{Existence of solutions for the nonlinear SIEs}\label{sec:NSDE}
In this part, we show the existence of strong solutions to \eqref{sys_NLS_R} by obtaining the weak-strong stability estimate.
\begin{proposition}\label{prop:exNLSDER}
There exists only one strong solution $[0,T]$ to the following nonlinear SIEs:
\begin{equation}\label{eq:NLSDER}
\left\{ \begin{array}{ll}
\displaystyle Y_t=Y_0+\int_0^t\left (\int_\mo \mb_{K(w(Y_s))}(y - Y_s)\nabla\varphi(Y_s-y) \rho_s(dy)\right )ds+\sqrt{2\sigma}B_t-K_t, \quad \mathcal{L}(Y_t) = \rho_t, \quad t > 0, & \\[3mm]
\displaystyle K_t=\int_0^t n (Y_s)\,d|K|_s, \quad |K|_t=\int_0^t \mb_{\pa\mo}(Y_s)\,d|K|_s,& \\[5mm]
\LL \left( Y_0\right)= \rho_0.&
\end{array} \right.
\end{equation}
Moreover, if $(Y^i_t)_{t \in (0,T]}, i=1,2$ are two solutions to \eqref{eq:NLSDER} with the initial data $(Y^i_0)$, respectively, and $\LL(Y^1_.) = \rho_.\in L^{\infty}(\mo\times[0,T])$ then we have the following stability estimate:
\[
\mathbb{P}\mbox{-}\esssup|Y^1_t-Y^2_t| \leq e^{C\int_0^t \|\rho_s\|_{L^1\cap L^\infty} ds} \,\mathbb{P}\mbox{-}\esssup|Y^1_0-Y^2_0| \quad \mbox{for} \quad t \in [0,T].
\]
\end{proposition}
\begin{proof}For the proof, we split it into three steps:

$\bullet$ {\bf Step A} ({\it Cauchy estimate}): Let $\e,\e',\eta,\eta'>0$, and consider the solutions $Y_.^{\e,\eta}$ and $Y_.^{\e',\eta'}$ to the regularized nonlinear SIEs \eqref{ApproNLSDER}. For notational simplicity, we set 

\[
V_t^{\e,\eta,\e',\eta'}(Y_t) := V^{\e,\eta}[\rho^{\e,\eta}_t](Y_t^{\e,\eta})-V^{\e',\eta'}[\rho^{\e',\eta'}_t](Y^{\e',\eta'}_t). 
\]
Applying Ito's formula yields
\begin{align*}%\label{eq:ItoR}
\begin{aligned}
\lt| Y_t^{\e,\eta}-Y^{\e',\eta'}_t \rt|^2&= 2\int_0^t \left\langle  Y_s^{\e,\eta}-Y^{\e',\eta'}_s ,V_s^{\e,\eta,\e',\eta'}(Y_s)\right\rangle \,ds\\
&\qquad -2\int_0^t \left\langle  Y_s^{\e,\eta}-Y^{\e',\eta'}_s ,n(Y_s^{\e,\eta})\right\rangle \, d|K^{\e,\eta}|_s-2\int_0^t \left\langle  Y_s^{\e',\eta'}-Y^{\e,\eta}_s ,n(Y_s^{\e',\eta'})\right\rangle \, d|K^{\e',\eta'}|_s\\
&\leq 2\int_0^t \left\langle  Y_s^{\e,\eta}-Y^{\e',\eta'}_s ,V_s^{\e,\eta,\e',\eta'}(Y_s)\right\rangle \,ds, 
\end{aligned}
\end{align*}	
since $(Y^{\e,\eta}_s-Y^{\e',\eta'}_s)\cdot n (Y^{\e,\eta}_s)\geq 0$, $d|K^{\e,\eta}|_s$ almost surely, due to the convexity of the domain $\mo$, see \eqref{eq_conv}. Thus it only remains to estimate the following term:
\[
J:=\left|V^{\e,\eta}[\rho^{\e,\eta}_s](Y_s^{\e,\eta})-V^{\e',\eta'}[\rho^{\e',\eta'}_s](Y^{\e',\eta'}_s)\right|.
\]
We decompose $J$ as
$$\begin{aligned}
J &\leq \left|V^{\e,\eta}[\rho^{\e,\eta}_s](Y_s^{\e,\eta})-V[\rho^{\e,\eta}_s](Y_s^{\e,\eta})\right|+ \left|V[\rho^{\e,\eta}_s](Y_s^{\e,\eta})-V[\rho^{\e',\eta'}_s](Y^{\e',\eta'}_s)\right|\cr
&\quad + \left|V[\rho^{\e',\eta'}_s](Y^{\e',\eta'}_s)-V^{\e',\eta'}[\rho^{\e',\eta'}_s](Y^{\e',\eta'}_s)\right|\cr
&=: J_1 + J_2 + J_3.
\end{aligned}$$
$\diamond$ Estimates of $J_1$ and $J_3$: Using \eqref{lem_diff1} and \eqref{lem_add}, we easily get for $\e,\eta \leq 1/2$
\[%\bq\label{est_j1}
\begin{split}
J_1&=\lt|\int\lt(\mb^{\e,\eta}_{K(w(Y_s^{\e,\eta}))}(y - Y_s^{\e,\eta})-\mb_{K(w(Y_s^{\e,\eta}))}(y - Y_s^{\e,\eta})\rt)\nabla\varphi(Y_s^{\e,\eta}-y) \rho_s^{\e,\eta}(dy) \rt|\\
&\leq\lt|\int\lt(\mb^{\e,\eta}_{K(w(Y_s^{\e,\eta}))}(y - Y_s^{\e,\eta})-\mb^{\e}_{K(w(Y_s^{\e,\eta}))}(y - Y_s^{\e,\eta})\rt)\nabla\varphi(Y_s^{\e,\eta}-y) \rho_s^{\e,\eta}(dy) \rt|\\
&\quad +\lt|\int\lt(\mb^{\e}_{K(w(Y_s^{\e,\eta}))}(y - Y_s^{\e,\eta})-\mb_{K(w(Y_s^{\e,\eta}))}(y - Y_s^{\e,\eta})\rt)\nabla\varphi(Y_s^{\e,\eta}-y) \rho_s^{\e,\eta}(dy) \rt|\\
&\leq C\|\nabla \varphi\|_{L^1 \cap L^\infty}\|\rho^{\e,\eta}_s\|_{L^\infty}(\e+\eta),
\end{split}
\]
where $C > 0$ is independent of $\e, \eta>0$. Employing the same argument as above, we estimate $J_3$ as
\[
J_3 \leq C\|\nabla \varphi\|_{L^\infty}\|\rho^{\e',\eta'}_s\|_{L^\infty}(\e'+\eta') \quad \mbox{for} \quad \e',\eta' < 1/2.
\]
$\diamond$ Estimate of $J_2$: It follows from Lemma \ref{lem:roparg} that
$$
J_2\leq C\|\rho_s^{\e,\eta}\|_{L^1\cap L^\infty}\IP\mbox{-}\esssup|Y^{\e,\eta}_s-Y_s^{\e',\eta'}|.
$$ 
Combining all the above estimates, we find
$$\begin{aligned}
|Y_t^{\e,\eta}-Y_t^{\e',\eta'}|^2 &\leq C\int_0^t \|\rho_s^{\e,\eta}\|_{L^1\cap L^\infty}\lt(\IP\mbox{-}\esssup|Y^{\e,\eta}_s-Y_s^{\e',\eta'}|\rt) |Y_s^{\e,\eta}-Y_s^{\e',\eta'}|\,ds\\
&\quad +  C\int_0^t \lt(\e+\eta+\e'+\eta'\rt)\lt(\|\rho_s^{\e,\eta}\|_{L^1\cap L^\infty}+\|\rho_s^{\e',\eta'}\|_{L^1\cap L^\infty}\rt) |Y_s^{\e,\eta}-Y_s^{\e',\eta'}|\,ds,
\end{aligned}$$
where $C >0$ is independent of $\e, \e',\eta,\eta' > 0$. Then using Lemma \ref{lem_gron} (ii) with $f(t) = |Y_t^{\e,\eta}-Y_t^{\e',\eta'}|$ and $p=2$ yields

\[%\bq\label{res_sta}
\sup_{0 \leq s\leq t}\IP\mbox{-}\esssup|Y^{\e,\eta}_s-Y_s^{\e',\eta'}| \leq C\lt(\e+\e'+\eta+\eta'\rt)\exp\lt(C\int_0^t \lt(\|\rho_s^{\e,\eta}\|_{L^1\cap L^\infty}+\|\rho_s^{\e',\eta'}\|_{L^1\cap L^\infty}\rt)ds \rt)
\]%\eq
for $t \in [0,T]$, where the constant $C$ is independent of $\e, \e', \eta$ and $\eta'$.

$\bullet$ {\bf Step B} ({\it Passing to the limit}): It follows from {\bf Step A} that there exists a limit process $(Y_t)_{t\in [0,T]}$ of the $(Y^{\e,\eta}_t)_{t\in [0,T]}$ as $\e,\eta\rightarrow 0$ in $L^1(\om,\mo\times [0,T])$. Using the fact that
$$
\sup_{t\in[0,T]}\mathcal{W}_{\infty}(\rho_t^{\e,\eta},\rho_t^{\e',\eta'})\leq \sup_{t\in[0,T]}\IP\mbox{-}\esssup|Y^{\e,\eta}_t-Y^{\e',\eta'}_t|,
$$
we also deduce that $(\rho_t^{\e,\eta})_{t\in[0,T]}$ is a Cauchy sequence in $\mc([0,T]; \mathcal{P}_1(\mo))$. Then, by completeness of this space, we define $\rho \in \mc([0,T]; \mathcal{P}_1(\mo))$ by $\rho_t := \lim_{\e,\eta \to 0} \rho_t^{\e,\eta}$ for $t \in [0,T]$. 

We now check $\eqref{eq:NLSDER}_2$. For this, we define $K_t$ as
\[
K_t:=Y_t-Y_0-\int_0^tV[\rho_s](Y_s)\,ds+\sqrt{2\sigma}B_t
\]
and recall 
\[
K_t^{\e,\eta}=Y^{\e,\eta}_t-Y_0-\int_0^t V^{\e,\eta}[\rho_s^{\e,\eta}](Y_s^{\e,\eta})\,ds+\sqrt{2\sigma}B_t.
\]
We then claim that
\bq\label{est_kk}
\E\lt[ \lt| \int_0^t \lt(V[\rho_s](Y_s)-V^{\e,\eta}[\rho_s^{\e,\eta}](Y_s^{\e,\eta})\rt)ds\rt| \rt] \to 0, \quad \mbox{as} \quad \e,\eta\rightarrow 0.
\eq
Let us assume that \eqref{est_kk} holds at the moment. We will give the proof of that in the last part of the step. We also notice that if \eqref{est_kk} holds, then the rest of the proof can be obtained by using the almost same argument as in \cite[Lemma 1.2]{Li-Si}. However, we provide the details of the proof for the reader's convenience and the completeness.

The convergence \eqref{est_kk} implies that $(K^{\e,\eta}_t)_{t\in [0,T]}$ converges to $(K_t)_{t\in [0,T]}$ in $L^{1}(\om,\mo\times [0,T])$ as $\e$ and $\eta$ go to $0$, and subsequently, for any $\phi \in \mc^{\infty}_c(\mo)$ with $0 \leq \phi \leq 1$, it deduces 
$$ 
\int_0^t\phi(Y_s)\,d|K|_s\leq \liminf_{\e,\eta \to 0}\int_0^t\phi(Y^{\e,\eta}_s)\,d|K^{\e,\eta}|_s=0.
$$
Taking an increasing sequence converging to $\mb_{\mo}$, we find
$$ 
\int_0^t\mb_{\mo}(Y_s)\,d|K|_s=0.
$$
This yields
$$ 
|K|_t=\int_0^td|K|_s=\int_0^t\lt(\mb_{\mo}(Y_s)+\mb_{\pa\mo}(Y_s)\rt)d|K|_s=\int_0^t\mb_{\pa\mo}(Y_s)\,d|K|_s.
$$
We now again use the convexity of the domain $\mo$ to get that for any $w\in \mo$ and $0 \leq \phi\in C([0,T])$ 
\[ 
\int_0^t   \left \langle Y^{\e,\eta}_s-w,n(Y^{\e,\eta}_s) \right \rangle\phi(s)\,d|K^{\e,\eta}|_s\geq 0, 
\]	
which can be rewritten as 
\[
\int_0^t  \phi(s) \left \langle Y^{\e,\eta}_s-w,dK^{\e,\eta}_s \right \rangle  \geq 0.
\]
We then let $\e,\eta \to 0$ to find that $d|K^{\e,\eta}|_s$ converges weakly (up to subsequence) to some measure $dm_s$ with $d|K|_s\leq dm_s$ and deduce 
	$$ \int_0^t  \phi(s) \left \langle Y_s-w,k_s \right \rangle \,dm_s\geq 0, $$	
	where we denoted $k_s$ a nonzero vector valued function such that $dK_s=k_sdm_s$ with $d|K|_s=|k_s|dm_s$ we obtain that 
\[
\left \langle Y_s-w,k_s \right \rangle\geq 0 \quad \mbox{for all} \quad w \in \overline\mo \quad m_s-\mbox{a.s.}.
\]
Then we find $k_s = |k_s| n(Y_s)$ due to the convexity of $\mo$. Hence we have
\[
K_t=\int_0^t n (Y_s)|k_s|\,dm_s=\int_0^t n(Y_s)\,d|K|_s. 
\]
{\it Proof of Claim \eqref{est_kk}}.- We first split \eqref{est_kk} into two parts:
$$\begin{aligned}
& \lt| \int_0^t \lt(V[\rho_s](Y_s)-V^{\e,\eta}[\rho_s^{\e,\eta}](Y_s^{\e,\eta})\rt)ds\rt|\cr
&\quad \leq  \lt| \int_0^t \lt(V[\rho_s](Y_s)-V[\rho_s^{\e,\eta}](Y_s^{\e,\eta})\rt)ds\rt|  + \lt| \int_0^t \lt(V[\rho^{\e,\eta}_s](Y^{\e,\eta}_s)-V^{\e,\eta}[\rho_s^{\e,\eta}](Y_s^{\e,\eta})\rt)ds\rt| \cr
&\quad =: L_1^{\e,\eta} + L_2^{\e,\eta}.
\end{aligned}$$
It follows from Lemma \ref{lem:roparg} that
$$\begin{aligned}
L_1^{\e,\eta} \leq C\|\nabla \varphi\|_{W^{1,\infty}}\|w\|_{Lip}\int_0^t\|\rho_s\|_{L^1 \cap L^\infty}\lt(\IP\mbox{-}\esssup|Y_s-Y_s^{\e,\eta}|\rt)\,ds \to 0,
\end{aligned}$$
as $\e,\eta \to 0$. Using similar arguments for the term $J_1$ in {\bf Step A}, $L_2$ can be estimated as 
\[
L_2^{\e,\eta} \leq C(\e+\eta)\|\nabla \varphi\|_{L^\infty}\int_0^t\|\rho^{\e,\eta}_s\|_{L^1\cap L^\infty}\,ds \quad \mbox{for} \quad  \e,\eta \leq 1/2.
\]
Thus $L_2^{\e,\eta} \to 0$ as $\e,\eta \to 0$, and this concludes the proof of claim.

$\bullet$ {\bf Step C} ({\it Stability estimate}): Using similar arguments as in {\bf Step A},  if $Y^i_t, i=1,2$ are two processes obtained as the above with the initial data $Y^i_0$, respectively, and $\LL(Y^i_t) = \rho^i_t, t \geq 0$ for $i=1,2$, using Lemma \ref{lem:roparg},  we easily find
$$\begin{aligned}
|Y_t^1-Y_t^2|^2 \leq C\int_0^t 2\|\rho_s\|_{L^1\cap L^\infty}\lt(\IP\mbox{-}\esssup|Y^1_s-Y_s^2|\rt) |Y_s^1-Y_s^2|\,ds\\
\end{aligned}$$
Similarly as in {\bf Step A}, we apply  Lemma \ref{lem_gron} (ii) with $p=2$ to conclude the stability estimate. 
\end{proof}

%\begin{remark}\label{gd_rem}If the domain $\mo$ is convex, one can simplify the above estimates by not introducing the function $\varphi_p$, but by using the arguments developed in \cite{Tan} for the existence of solutions of a stochastic differential equation for a reflecting diffusion process.
%\end{remark}

%%%%%%%%%%%%%%%%%%%%%%%%%%%%%%%%%%%%%%%%%%%%%%%%%%%%%%%%%%%%%%%%%%%%%%%%%%%%%%%%%%%%%%%%%%%%%%%%%%%%%%%%%%%%%%%%%%%%%%%%%%%%%%%%%%%%%%%%%%%%%%
%
%
%         \subsection{Proof of Theorem \ref{thm_SDE}}
%
%
%%%%%%%%%%%%%%%%%%%%%%%%%%%%%%%%%%%%%%%%%%%%%%%%%%%%%%%%%%%%%%%%%%%%%%%%%%%%%%%%%%%%%%%%%%%%%%%%%%%%%%%%%%%%%%%%%%%%%%%%%%%%%%%%%%%%%%%%%%%%%%%%

\subsection{Proof of Theorem \ref{thm_SDE}}
The existence and uniqueness of strong solutions to the nonlinear SIEs \eqref{sys_NLS_R} just follows from Proposition \ref{prop:exNLSDER}. For the existence of weak solutions for the equation \eqref{1_pde_R}, just take any test function $\phi\in\mc^{\infty}(\mo)$ with $\left\langle \nabla \phi(x),n(x) \right\rangle=0 $ on $\partial \mo$ and apply Ito's formula to the solution to \eqref{sys_NLS_R}, then find that its time marginals $(\rho_t)_{t\in [0,T]}$ solves the equation \eqref{1_pde_R} in the distributional sense. For the uniqueness of solutions, we move the stability estimate of solutions for SIEs obtained in Proposition \ref{prop:exNLSDER} on to some stability estimate for the corresponding PDE. In order to do so, we use the fact that for any solutions to \eqref{1_pde_R} can be seen as the time marginals of some solutions to \eqref{eq:NLSDER}. Let $(\tilde{\rho}_t)_{t \ge 0}$ be a weak solution to \eqref{1_pde_R} with the initial data $\tilde{\rho}_0\in \mathcal{P}_1(\mo)$ and $(\rho_t)_{t \ge 0}$ be another weak solution to \eqref{1_pde_R} with the initial data $\rho_0\in (L^1_+\cap L^{\infty})(\mo) \cap \mathcal{P}_1(\mo)$ such that $\rho\in L^\infty(0,T; (L^1_+ \cap L^{\infty})(\mo))\cap \mc([0,T];\pp_1(\mo))$. Then, by Lemma \ref{lem:app_b}, we can find a probability space $( \Omega,\mathbb{P},\left (\mathcal{F}_t \right )_{t\geq 0},\mathcal{F})$, a Brownian motion $(B_t)_{t \ge 0}$ on that basis and a process $(X_t)_{t \ge 0}$ solution to \eqref{eq:NLSDER}, which has the time marginal $\tilde{\rho}_t$ at any time $t \ge 0$. On that probability space, let $Y_0$ be a random variable on $\mo$ with the law $\rho_0$ independent of $(B_t)_{t\geq 0}$ such that 
\[
\mathcal{W}_\infty(\rho_0,\tilde{\rho}_0)=\IP\mbox{-}\esssup|Y_0-X_0|.
\]
Note that it is known that such an optimal coupling exists when $\rho_0$ is absolutely continuous with respect to the Lebesgue measure, see \cite{CPJ}. 
On the other hand, since $\rho$ has a sufficient regularity for the velocity field to be Lipschitz (see the proof of Proposition \ref{lem:EstLinf}), the standard theory on linear SDEs allows to build some stochastic process $(Y_t)_{t\geq 0}$ which is a solution to \eqref{eq:NLSDER} with the initial condition $Y_0$, and same Brownian motion as exhibited in the beginning of this step, such that its marginal at time $t$ is $\rho_t$. Hence, by definition of $\mathcal{W}_\infty$ distance and Proposition \ref{prop:exNLSDER}, it is straightforward to deduce that
\[
\mathcal{W}_\infty(\rho_t,\tilde{\rho}_t)\leq \IP\mbox{-}\esssup |Y_t-X_t|\leq \IP\mbox{-}\esssup |Y_0-X_0|e^{\int_0^t\|\rho_s\|_{L^1\cap L^{\infty}} ds}=\mathcal{W}_\infty(\rho_0,\tilde{\rho}_0)e^{\int_0^t\|\rho_s\|_{L^1\cap L^{\infty}} ds},
\] 
from which the uniqueness of solutions to \eqref{1_pde_R} follows. This completes the proof.
\begin{remark} It is worth noticing that we are not able at this point to extend this result to the case where $\mo$ is not convex but only satisfies the exterior sphere condition \eqref{eq_extc1}. This is due to the fact that we can only obtain the weak-strong stability estimate in the $\mathcal{W}_{\infty}$ metric. That is why we have to estimate the $\IP$ essential supremum of the distance between two regularized solutions. If the domain $\mo$ only satisfies the condition \eqref{eq_extc1}, then we need to use the similar strategy as in \cite{Li-Si}, together with approximating the $\IP$-essential supremum by $\E[|\cdot|^p]^{1/p}$ with $p \geq 1$. However, this gives a $p$-dependent constant in the estimates and it cannot be removed. Thus our arguments fail to the case in which the domain $\mo$ only satisfies the condition \eqref{eq_extc1}.
\end{remark}
	
%Then thanks some Theorem due to Figalli \cite{Figalli}, or Hauray and Fournier \cite{FouHau}, (actually these theorems deal with more singular kernels than the one considered here, so that their results do not hold directly but the strategy would apply all the same. For the sake of the compactness of this paper, we do not prove this not essential result, in view of the existing literature) it is possible to find 

%%%%%%%%%%%%%%%%%%%%%%%%%%%%%%%%%%%%%%%%%%%%%%%%%%%%%%%%%%%%%%%%%%%%%%%%%%%%%%%%%%%%%%%%%%%%%%%%%%%%%%%%%%%%%%%%%%%%%%%%%%%%%%%%%%%%%%%%%%%%%%
%
%
%         Section: Mean-field limits
%
%
%%%%%%%%%%%%%%%%%%%%%%%%%%%%%%%%%%%%%%%%%%%%%%%%%%%%%%%%%%%%%%%%%%%%%%%%%%%%%%%%%%%%%%%%%%%%%%%%%%%%%%%%%%%%%%%%%%%%%%%%%%%%%%%%%%%%%%%%%%%%%%%%
\section{Propagation of chaos}\label{sec:mf}
\subsection{Law of large numbers like estimates}
In this subsection, we provide types of the law of large numbers estimates which relies on the nice property of our communication function observed in Lemma \ref{lem_est2}. 
\begin{lemma}\label{lem_(De)-Poiss}
Let $m,N\in \mathbb{N}$ with $N\geq (2m)^2$ and $Y_1, \cdots, Y_N$ be $N$ i.i.d. random variables with the law $\rho\in \mathcal{P}(\mo)$, $Y$ be independent of $Y_1, \cdots, Y_N$, and let $\rho^N$ be the associated empirical measure $\rho^N=\frac{1}{N}\sum_{i=1}^N\delta_{Y_i}$. Then we have
\[
\mathbb{E}\lt[\sup_{i=1,\cdots,N}\sup_{u\geq 0}\left|\int_\mo \mb_{\Theta(w(Y_i))^{u,+}}(y-Y_i)(\rho^N-\rho)(dy)\right|^{2m}\rt]^{\frac{1}{2m}} \leq 8e^{2m} N^{-\frac{1}{2}+\frm}.
\]
\end{lemma}
\begin{proof}
Let $(Y_n)_{n\in\mathbb{N}}$ be a sequence of independent random variables with the law $\rho\in \mathcal{P}(\mo)$, and $K_N$ be a Poisson random variable of parameter $N$ independent of $(Y_n)_{n\in\mathbb{N}}$. Define $\varrho^N $ by the following random measure 
\[
\varrho^N := \sum_{i=1}^{K_N}\delta_{Y_i}. 
\]
 Then $\varrho^N $ is a Poisson random measure of the intensity measure $N\rho$. It is straightforward to get $\left\|\varrho^N -N\rho^N\right\|_{TV} = |K_N-N|$, where $\| \cdot \|_{TV}$ represents the total variation of signed measures. This yields
 
 \[
 \begin{aligned}
 &\mathbb{E}\lt[\sup_{i=1,\cdots,N}\sup_{u\geq 0}\left|\int_\mo \mb_{\Theta(w(Y_i))^{u,+}}(y-Y_i)(\rho^N-\rho)(dy)\right|^{2m}\rt]^{\frm}\cr
 &\quad \leq \mathbb{E}\lt[\sum_{i=1}^N\lt(\sup_{u\geq 0}\left|\int_\mo \mb_{\Theta(w(Y_i))^{u,+}}(y-Y_i)(\rho^N-\rho)(dy)\right|\rt)^{2m}\rt]^{\frm}\\
 &\quad \leq \mathbb{E}\lt[N^{-2m}\sum_{i=1}^N\sup_{u\geq 0}\lt(\left|\int_{\mo}\mb_{\Theta(w(Y_i))^{u,+}}(y-Y_i)(N\rho^N-\varrho^N)(dy)+\mathcal{M}^{N,Y_i}_u\right|\rt)^{2m}\rt]^{\frm}\\
 &\quad \leq 	2\mathbb{E}\lt[N^{-2m}\sum_{i=1}^N\lt(\sup_{u\geq 0}\lt|\int_\mo \mb_{\Theta(w(Y_i))^{u,+}}(y-Y_i)(N\rho^N-\varrho^N)(dy)\rt|^{2m}+\sup_{u\geq 0}\lt|\mathcal{M}_u^{N,Y_i}\rt|^{2m}\rt)\rt]^{\frm}\\
 &\quad \leq 2\mathbb{E}\lt[N^{-2m}\sum_{i=1}^N\lt(\left\|\varrho^N -N\rho^N\right\|_{TV}^{2m}+\sup_{u\geq 0}\lt|\mathcal{M}_u^{N,Y_i}\rt|^{2m}\rt)\rt]^{\frm},
 \end{aligned}
 \]
where
\[
\mathcal{M}^{N,Y_i}_u:=\int_\mo \mb_{\Theta(w(Y_i))^{u,+}}(y-Y_i)\overline{\varrho}^N(dy) \quad \mbox{with} \quad \overline{\varrho}^N := \varrho^N  - N\rho.
\]
Since the $(Y_i)_{i=1,\cdots,N}$ are i.i.d, we find
\begin{equation}
\label{eq:est_4.1}
\begin{aligned}
 &\mathbb{E}\lt[\sup_{i=1,\cdots,N}\sup_{u\geq 0}\left|\int_\mo \mb_{\Theta(w(Y_i))^{u,+}}(y-Y_i)(\rho^N-\rho)(dy)\right|^{2m}\rt]^{\frm}\cr
 &\qquad \leq 	2N^{\frac{-2m+1}{2m}}\lt(\E\lt[|K_N-N|^{2m}\rt]^{\frm}+\E\lt[\sup_{u\geq 0}\lt|\mathcal{M}_u^{N,Y_1}\rt|^{2m} \rt]^{\frm}\rt).
 \end{aligned}
\end{equation}
We next observe that  $(\mathcal{M}_u^{N,Y_1})_{u\geq 0}$ conditioned to $Y_1$ is a martingale. Indeed, for $a \in \mo$, we define the filtration $(\mathcal{F}^a_{u})_{u\geq 0}$ as
\[
\mathcal{F}^a_{u}=\sigma\left\{\int_\mo h(y) \varrho^N(dy) \ | \ supp \ h \subseteq \Theta(w(a))^{r,+} +a, \, r\leq u \right\},
\]
Then, for $s > u$, we find 
$$\begin{aligned}
&\mathbb{E}\left [ \int_\mo \mb_{\Theta(w(a))^{s,+})}(y-a)\overline{\varrho}^N(dy) \,\big| \,\mathcal{F}^a_{u} \right ]\cr
&\qquad =\mathbb{E}\left [ \int_\mo \mb_{\Theta(w(a))^{s,+}\setminus\Theta(w(a))^{u,+}}(y-a)\overline{\varrho}^N(dy) \, \big| \, \mathcal{F}^a_{u} \right ]+\mathbb{E}\left [ \int_\mo \mb_{\Theta(w(a))^{u,+}}(y-a)\overline{\varrho}^N(dy) \, \big| \, \mathcal{F}^a_{u} \right]
\end{aligned}$$
due to the linearity of conditional expectation and the following property of indicator function
\[
\mb_{A} = \mb_{A\setminus B} + \mb_B \quad \mbox{for} \quad A, B \subset \R^d \mbox{ and } B \subseteq A.
\] 
Since $\varrho^N (A)$ is independent of $\varrho^N (B)$ for disjoint sets $A,B\subset \mathbb{R}^d$, by the definition of Poisson random measure (see \cite{Cin}), 
\[
\int_\mo\mb_{\Theta(w(a))^{s,+}\setminus\Theta(w(a))^{u,+}}(y-a)\,\overline{\varrho}^N(dy)=\overline{\varrho}^N\lt(\lt(\Theta(w(a))^{s,+}\setminus\Theta(w(a))^{u,+}\rt)+a\rt),
\]
is independent of all $\varrho^N (\Theta(w(a))^{s,+}\setminus\Theta(w(a))^{r,+} )$ for $r\leq u$, thus it is also independent of $\mathcal{F}^a_{u}$. This yields
\[
\mathbb{E}\left [ \int_\mo \mb_{\Theta(w(a))^{s,+}\setminus\Theta(w(a))^{u,+}}(y-a)\,\overline{\varrho}^N(dy) \ | \ \mathcal{F}^{a}_{u} \right ]=\mathbb{E}\left [ \int_\mo \mb_{\Theta(w(a))^{s,+}\setminus\Theta(w(a))^{u,+}}(y-a)\,\overline{\varrho}^N(dy)  \right ]=0.
\]
On the other hand, since
$
\int_\mo \mb_{\Theta(w(a))^{u,+}}(y-a)\overline{\varrho}^N(dy) \mbox{ is $\mathcal{F}^{a}_{u}$-measurable}$, we deduce 
\[
\mathbb{E}\left [ \int_\mo \mb_{\Theta(w(a))^{u,+}}(y-a)\,\overline{\varrho}^N(dy) \ | \ \mathcal{F}^{a}_{u} \right ]=\int_\mo \mb_{\Theta(w(a))^{u,+}}(y-a)\,\overline{\varrho}^N(dy), 
\]
and $(\mathcal{M}^{N,a}_u)_{u\geq 0}$ is a martingale. We now use Doob's inequality to obtain
	$$
	\begin{aligned}
				\mathbb{E}\lt[\sup_{u\geq 0}\left|\mathcal{M}^{N,a}_u\right|^{2m}\rt]&\leq \lt(\frac{2m}{2m-1}\rt)^{2m} \mathbb{E}\lt[\left|\mathcal{M}^{N,a}_\infty\right|^{2m}\rt]\\
				&=  \lt(\frac{2m}{2m-1}\rt)^{2m} \mathbb{E}\lt[ \lt| \varrho_N(\mo)-N\rho(\mo)  \rt|^{2m}\rt]  \\
				&= \lt(\frac{2m}{2m-1}\rt)^{2m} \mathbb{E}\lt[ \lt| K_N-N  \rt|^{2m}\rt],
	\end{aligned}
	$$
We next use a standard property of the conditional expectation to get
$$
\begin{aligned}
\mathbb{E}\lt[\sup_{u\geq 0}\lt|\mathcal{M}_u^{N,1}\rt|^{2m}\rt]&=	\mathbb{E}\lt[\mathbb{E}\lt[\sup_{u\geq 0}\left|\int_\mo \mb_{\Theta(w(Y_1))^{u,+}}(y-Y_1)(\rho^N-\rho)(dy)\right|^{2m}\ | \ Y_1 \rt]\rt]\\
&\leq \lt(\frac{2m}{2m-1}\rt)^{2m} \mathbb{E}\lt[ \lt| K_N-N  \rt|^{2m}\rt].
\end{aligned}
$$
Coming back to \eqref{eq:est_4.1}, we then find
$$\begin{aligned}
&\mathbb{E}\lt[\sup_{i=1,\cdots,N}\sup_{u\geq 0}\left|\int_\mo \mb_{\Theta(w(Y_i))^{u,+}}(y-Y_i)(\rho^N-\rho)(dy)\right|^{2m}\rt]^{\frm} \cr
&\qquad \leq 2\lt(1+\frac{2m}{2m-1}\rt) N^{-1+\frm} \E\lt[|K_N-N|^{2m}\rt]^{\frm}.
\end{aligned}$$
Note that $K_N$ is the Poisson($N$)-distributed random variable, thus it holds
\[
\E\lt[\exp\lt(2m\frac{|K_N-N|}{\sqrt{N}}\rt) \rt]\leq 2\exp\lt(N\lt(e^\frac{2m}{\sqrt{N}}-1-\frac{2m}{\sqrt{N}}\rt)\rt)\leq 2e^{(2m)^2}, 
\]
due to 
\[
N\lt(e^{\frac{2m}{\sqrt{N}}} - 1 - \frac{2m}{\sqrt{N}}\rt) = (2m)^2 \int_0^1 (1-t)e^{\frac{2m}{\sqrt{N}}t}\,dt \leq (2m)^2 \frac{\sqrt{N}}{2m}\lt(e^{\frac{2m}{\sqrt{N}}} - 1\rt) \leq (2m)^2,
\]
where we used the fact that the function $(1-t)e^{\frac{2m}{\sqrt{N}}t}$ on $[0,1]$ is bounded by $\frac{\sqrt{N}}{2m}(e^{\frac{2m}{\sqrt{N}}} - 1)$ and $e^{x-1} \leq x$ for $0 \leq x \leq 1$.
This yields 

\[
\frac{1}{N^m}\E\bigl[ |K_N-N|^{2m} \bigr]\leq \E\lt[e^{2m\frac{|K_N-N|}{\sqrt{N}}} \rt]\leq 2e^{(2m)^2}.
\]
\end{proof}

\begin{lemma}
	\label{lem:LLN}
	Let $m,N\in \mathbb{N}$ with $N\geq (2m)^2$ and $Y_1,\cdots,Y_N$ be $N$ i.i.d. random variables with the law $\rho\in \mathcal{P}(\mo)$. Then we have 
	\[
	\mathbb{E}\lt[ \sup_{i=1,\cdots,N}\lt| \int_\mo \mb_{K(w(Y_i))}(y-Y_i)\nabla\varphi(Y_i-y)(\rho^N(dy)-\rho(dy))\rt|^{2m} \rt]^{\frm} \leq \|\nabla \varphi\|_{L^{\infty}} ( 2m!)^\frm\sqrt{8m}N^{-\frac{1}{2}+\frm},
	\] 
	where $\rho^N=\frac{1}{N}\sum_{i=1}^N\delta_{Y_i}$.
\end{lemma}
\begin{proof}Let $\mathcal{M}$ be a set of $N$-dimensional multi-index of order $2m$, i.e., 
\[
\mathcal{M}:=\left \{ \alpha=(\alpha_1,\cdots,\alpha_N)\in \mathbb{N}^N : |\alpha| = \sum_{i=1}^N \alpha_i=2m\right \},
\]
and define $\mathcal{M}_1 := \lt\{ (\alpha_j)_{j}\in 	\mathcal{M} \, : \,   \alpha_j\neq 1, \,\forall \, j=1,\cdots,N \rt\}$. Let us denote by $\ell_\alpha:=\lt| \{i=1,\cdots,N \, : \, \alpha_i >0 \} \rt|$ with $\alpha \in \mathcal{M}_1$. It is clear that
$\ell_\alpha \leq m$. We now consider two functions $\phi_\alpha : \{1,\cdots,\ell_\alpha \} \to \{1,\cdots, N\}$ and $\psi_\alpha : \{1,\cdots,\ell_\alpha \} \to \{2,\cdots, 2m\}$ with $\alpha \in \mathcal{M}_1$ defined by
\[
\phi_\alpha (i)=\mbox{``the index of the $i$-th nonzero component of } \alpha"
\]
and
\[
\psi_\alpha (i)=\mbox{``the value of the $i$-th nonzero component of } \alpha",
\]
respectively.
%Denote now  the two following applications defined respectively by
%$$ \phi_{\boldsymbol{\alpha}}:\{1,\cdots,\ell_{\boldsymbol{\alpha}} \}\ni i\mapsto \phi_{\boldsymbol{\alpha}}(i)=\mbox{"the index of the i-th no nul component of} \, \boldsymbol{\alpha}" \in\{1,\cdots,N\},$$
%and
%$$\psi_{\boldsymbol{\alpha}}:\{1,\cdots,\ell_{\boldsymbol{\alpha}}\}\ni i\mapsto \psi_{\boldsymbol{\alpha}}(i)=\mbox{"the value of the i-th no nul component of} \, \boldsymbol{\alpha}" \in\{2,\cdots,2m\}. $$
Then we get that
$$\Phi:\mathcal{M}_1\ni \alpha\mapsto (\phi_\alpha,\psi_\alpha)\in \bigcup_{k=1}^m \{1,\cdots,N\}^{\{1,\cdots,k\}}\times \{1,\cdots,2m\}^{\{1,\cdots,k\}},$$
 is injective and thus
 $$\lt|\mathcal{M}_1 \rt|\leq \sum_{k=1}^m N^k(2m)^k \leq 2(2m)^m N^m. $$
For notational simplicity, we now set for a fixed $i \in \{1,\cdots, N\}$
\[
h^i_k:=\nabla \varphi(Y_i-Y_k)\mb_{K(w(Y_i))}(Y_k-Y_i)-\int_{\mo}\nabla \varphi(Y_i-y)\mb_{K(w(Y_i))}(y-Y_i)\rho(dy) \quad \mbox{and} \quad \mathbf{h}^i:=(h^i_1,\cdots,h^i_N).
\]
A straightforward computation yields
$$\begin{aligned}
\mathbb{E}&\lt[ \sup_{i=1,\cdots,N}\lt| \int_\mo \mb_{K(w(Y_i))}(y-Y_i)\nabla\varphi(Y_i-y)(\rho^N(dy)-\rho(dy))\rt|^{2m} \rt]^{\frm}\\
&\quad \leq  	\mathbb{E}\lt[ \lt(\sum_{i=1}^N \lt(\int_\mo \nabla \varphi(Y_i-y)\mb_{K(w(Y_i))}(y-Y_i)(\rho^N(dy)-\rho(dy))\rt)^{2m}\rt) \rt]^{\frm}\\
&\quad  =	\mathbb{E}\lt[ \lt(\sum_{i=1}^N N^{-2m}\lt(\sum_{j=1}^Nh^i_j\rt)^{2m}\rt)\rt]^{\frm}\\
& \quad = \lt(	\sum_{i=1}^N N^{-2m}\mathbb{E}\lt[ \lt(\sum_{j=1}^Nh_j^i\rt)^{2m}\rt]\rt)^{\frm}.
\end{aligned}$$
Then it follows from the law of total expectation that
$$\begin{aligned}
\mathbb{E}\lt[ \lt(\sum_{j=1}^N h^i_j\rt)^{2m}  \rt]&=\E\lt[ \mathbb{E}\lt[ \lt(\sum_{j=1}^Nh_j^i\rt)^{2m} \mid Y_i \rt]\rt]=\E\lt[ \mathbb{E}\lt[\sum_{\alpha\in \mathcal{M}}\binom{2m}{\alpha} (\mathbf{h}^i)^\alpha \mid Y_i \rt]\rt].
\end{aligned}$$
Note that $(h_k^i)^{\alpha_k}_{k=1,\cdots,N}$ conditioned on $Y_i$ are independent since $(Y_k)_{k=1,\cdots,N}$ are independent. Thus we find
\[
\mathbb{E}\lt[ \lt(\sum_{j=1}^N h^i_j\rt)^{2m}  \rt]=\sum_{\alpha\in \mathcal{M}}\binom{2m}{\alpha}\E\lt[ \prod_{j=1}^N\mathbb{E}\lt[ (h^i_j)^{\alpha_j} \mid Y_i \rt]\rt].
\]
We also notice that
\bq\label{est_ze}
\E\lt[\nabla \varphi(Y_i-Y_j)\mb_{K(w(Y_i))}(Y_j-Y_i)\, \mid Y_i\rt]=\int_{\mo}\nabla \varphi(Y_i-y)\mb_{K(w(Y_i))}(y-Y_i)\,\rho(dy), \quad \mbox{i.e.,} \quad \E\lt[h_j^i\, \mid Y_i\rt]=0.
\eq
This implies 
\[
\E\lt[ (\mathbf{h}^i)^\alpha \,|\, Y_i\rt]=0 \quad \mbox{for} \quad \alpha\notin \mathcal{M}_1.
\]
On the other hand, we obtain 
\[
\E\lt[ (\mathbf{h}^i)^\alpha \,|\,  Y_i\rt]\leq \bigl(2 \|\nabla \varphi \|_{W^{1,\infty}}  \bigr)^{2m},
\]
due to $|h_k^i|\leq 2$. This and together with \eqref{est_ze} yields
\[
\mathbb{E}\lt[ \lt(\sum_{j=1}^N h^i_j\rt)^{2m}  \rt]\leq \sup_{\alpha \in \mathcal{M}}\binom{2m}{\alpha}   \bigl(2 \|\nabla \varphi \|_{W^{1,\infty}}  \bigr)^{2m} \lt|\mathcal{M}_1\rt| \leq  \sup_{\alpha \in \mathcal{M}}\binom{2m}{\alpha}(8\|\nabla \varphi \|_{W^{1,\infty}}^2m)^mN^m, 
\]
and the result follows.

\end{proof}
We also provide a kind of weak-strong Lipschitz estimate for the velocity fields in the lemma below.	
\begin{lemma}
	\label{lem:coup_T}
Let $N\in \mathbb{N}$ and $X_1,\cdots, X_N$ be $N$ exchangeable random variables on $\mo$. Define $\mu^N=\frac{1}{N}\sum_{i=1}^N\delta_{X_i}$ the empirical measure associated. Let  $Y_1, \cdots, Y_N$ be $N$ i.i.d random variables on $\mo$ with the law $\rho\in (L^1_+\cap L^{\infty})(\mo) \cap \pp_1(\mo)$ and $\rho^N=\frac{1}{N}\sum_{i=1}^N\delta_{Y_i}$. Then there exists a constant $C > 0$ which is independent of $N$ and a random variable $H_N$ such that 
\[
 \sup_{i=1,\cdots,N} \left| V[\mu^N](X_i) - V[\rho](Y_i)\right| \leq C\left \| \rho\right \|_{L^1 \cap L^\infty}\sup_{i=1,\cdots,N}|X_i-Y_i|+ H_N,
\]
where $V$ is given in \eqref{def_V} and the set $K$ in $V$ satisfies the assumption ${\bf (H1)}$-${\bf (H2)}$. Here $H_N$ is given by
$$
H_N:=\|\nabla \varphi\|_{L^{\infty}} \sup_{i=1,\cdots,N}\sup_{u\geq 0}\left |\int_\mo \mb_{\Theta(w(Y_i))^{u,+}}(y-Y_i)(\rho^N-\rho)(dy)  \right |+ \sup_{i=1,\cdots,N}\left|V [\rho^N](Y_i)-V[\rho](Y_i)\right|
$$
and satisfies
\[
\E\bigl[ H_N^{2m}  \bigr]^\frm\leq C_m N^{-\frac{1}{2}+\frm},
\]
for $m\in \mathbb{N}$ such that $(2m)^2\leq N$, where $C_m$ is a positive constant, specified in the proof, depending on $m$, but not $N$. 
\end{lemma}
\begin{proof} For notational simplicity, we denote by $A_N=\sup_{j=1,\cdots,N}|X_j-Y_j|$. For any  $i=1,\cdots,N$, we have
$$\begin{aligned}
& \left| V[\mu^N](X_i) - V[\rho](Y_i)\right| \cr
& \quad \leq  \left|V[\rho^N](Y_i)-V[\mu^N](X_i)\right| + \left|V[\rho^N](Y_i)-V[\rho](Y_i)\right| \\
& \quad \leq \left| \frac{1}{N}\sum_{j=1}^N\Bigl(\mb_{K(w(X_i))}(X_j-X_i)\nabla\varphi(X_i - X_j)-\mb_{K(w(Y_i))}(Y_j-Y_i)\nabla\varphi(Y_i-Y_j)\Bigr)\right|\\
&\qquad +\left|\frac{1}{N}\sum_{j =1}^N\mb_{K(w(Y_i))}(Y_j-Y_i)\nabla\varphi(Y_i-Y_j)-V[\rho](Y_i)\right| \cr
&\quad \leq \left| \frac{1}{N}\sum_{j=1}^N\mb_{K(w(X_i))}(X_j-X_i)\Bigl(\nabla\varphi(X_i-X_j)-\nabla\varphi(Y_i-Y_j)\Bigr)\right| \cr
&\qquad +  \left| \frac{1}{N}\sum_{j=1}^N\Bigl(\mb_{K(w(X_i))}(X_j-X_i)-\mb_{K(w(X_i))}(Y_j-Y_i)\Bigr)\nabla\varphi(Y_i-Y_j)\right| \cr
&\qquad + \lt|\frac{1}{N}\sum_{j=1}^N\Bigl(\mb_{K(w(X_i))}(Y_j-Y_i)-\mb_{K(w(Y_i))}(Y_j-Y_i)\Bigr)\nabla\varphi(Y_i-Y_j)\rt| +  \left|V [\rho^N](Y_i)-V[\rho](Y_i)\right| \cr
&\quad =: I^i_1+I^i_2+I^i_3+I^i_4.
\end{aligned}$$	
$\diamond$ Estimate of $I^i_1$: We easily find 
\[
I_1^i\leq \|\nabla\varphi\|_{W^{1,\infty}}\lt(|X_i-Y_i|+\frac{1}{N}\sum_{j=1}^N|X_j-Y_j|\rt)\leq 2 \|\nabla\varphi\|_{W^{1,\infty}}A_N \leq C\|\rho\|_{L^1 \cap L^\infty}A_N.
\]

\noindent $\diamond$ Estimate of $I^i_2$: It follows from Lemma \ref{lem_est2}, the assumptions $\bf{(H2)}$ (i) and (iv) that
\[
\begin{aligned}
\left | \mb_{K(w(X_i))}(X_j-X_i)-\mb_{K(w(X_i))}(Y_j-Y_i) \right |&\leq \mb_{\partial^{2|X_i-Y_i|}K(w(X_i))}(Y_j-Y_i)+\mb_{\partial^{2|X_j-Y_j|}K(w(X_i))}(Y_j-Y_i)\\
&\leq \mb_{\partial^{2A_N}K(w(X_i))}(Y_j-Y_i)\\
&\leq \mb_{\Theta(w(Y_i))^{2A_N+2\|w\|_{Lip}A_N,+}}(Y_j-Y_i)\\
& \leq \mb_{\Theta(w(Y_i))^{CA_N,+}}(Y_j-Y_i).
\end{aligned}
\]
Then, thanks to $\bf{(H2)}$ (ii), we obtain
$$\begin{aligned}
&\frac{1}{N}\sum_{j=1 }^N\mb_{\Theta(w(Y_i))^{CA_N,+}}(Y_j-Y_i)\cr
&\quad = \int_\mo \mb_{\Theta(w(Y_i))^{CA_N,+}}(y-Y_i)\rho^{N}(dy)\\
&\quad \leq \int_\mo \mb_{\Theta(w(Y_i))^{CA_N,+}}(y-Y_i)\rho(dy)+\left |\int_\mo \mb_{\Theta(w(Y_i))^{CA_N,+}}(y-Y_i)(\rho^N-\rho)(dy)  \right |\\
&\quad \leq C\left \| \rho \right \|_{L^1 \cap L^\infty}A_N+\sup_{u\geq 0}\left |\int_\mo \mb_{\Theta(w(Y_i))^{u,+}}(y-Y_i)(\rho^N-\rho)(dy)  \right |.
\end{aligned}$$
$\diamond$ Estimate of $I^i_3$: Using the assumptions $\bf{(H2)}$ (ii) and (iii) together with similar estimates as the above, we get
\[
\begin{aligned}
I^i_3&\leq \|\nabla\varphi\|_{L^\infty}\lt|\int_{\mo}\mb_{K(w(X_i))\Delta K(w(Y_i))}(y-Y_i)\rho^N(dy)\rt|\\
&\leq \|\nabla\varphi\|_{L^\infty}\lt|\int_{\mo}\mb_{\Theta(w(Y_i))^{C\|w\|_{Lip}|X_i-Y_i|,+}}(y-Y_i)\rho^N(dy)\rt|\\
&\leq \|\nabla\varphi\|_{L^\infty}\int_{\mo}\mb_{\Theta(w(Y_i))^{C\|w\|_{Lip}A_N,+}}(y-Y_i)\rho(dy)\cr
&\quad +\|\nabla\varphi\|_{L^\infty}	 \lt|\int_{\mo}\mb_{\Theta(w(Y_i))^{C\|w\|_{Lip}A_N,+}}(y-Y_i)(\rho^N-\rho)(dy)\rt| \\
&\leq \|\nabla\varphi\|_{L^\infty}\lt(  C\|w\|_{Lip}\|\rho\|_{L^1\cap L^\infty}A_N+ \sup_{u\geq 0}\lt|\int_{\mo}\mb_{\Theta(w(Y_i))^{u,+}}(y-Y_i)(\rho^N-\rho)(dy) \rt| \rt).
\end{aligned}
\]
$\diamond$ Estimate of $I^i_4$: It just follows from Lemma \ref{lem:LLN} that
\[
\E\lt[\sup_{i=1,\cdots,N} \bigl(I^i_4\bigr)^{2m}\rt]^\frm \leq \|\nabla \varphi\|_{L^{\infty}} ( 2m!)^\frm\sqrt{8m}N^{-\frac{1}{2}+\frm}.
\]
We then combine the above estimates and take the supremum over $i=1,\cdots,N$ to find 
\[
\sup_{i=1,\cdots,N} \left| V[\mu^N](X_i) - V[\rho](Y_i)\right| \leq C\left \| \rho\right \|_{L^1 \cap L^\infty}A_N+ H_N.
\]
Finally, we use Lemma \ref{lem_(De)-Poiss} for the second term on the right hand side of the above inequality to conclude the desired result.
\end{proof}

\subsection{Proof of Theorem \ref{thm:PC_T}}
It follows from Theorems \ref{thm_SIE} and \ref{thm_SDE} that there exist a weak solution to \eqref{sys_sde_R} and a unique pathwise solution to \eqref{sys_NLS_R} on the time interval $[0,T]$ for some $T > 0$. This implies that we are able to define solutions for those two equations on the same probability space with the same initial condition and Brownian motion. On that probability space, we define $\mu_t^N,\rho_t^N$ the empirical measures associated to system \eqref{sys_sde_R} and \eqref{sys_NLS_R} by
\[
\mu_t^N=\frac{1}{N}\sum_{i=1}^N\delta_{X_t^i} \quad \mbox{and} \quad \rho_t^N=\frac{1}{N}\sum_{i=1}^N\delta_{Y_t^i},
\]
respectively. Using the similar argument as before,  we get
$$\begin{aligned}
|X_i^t-Y_i^t|^2 &\leq  2\int_0^t \left\langle X^{i}_s-Y^{i}_s,V[\mu_s^N](X_s^i)-V[\rho_s](Y_s^i)\right\rangle \,ds\\
&\quad -2\int_0^t \left\langle X^{i}_s-Y^{i}_s, n(X_s^i)\right\rangle d|K^i|_s -2\int_0^t \left\langle Y^{i}_s-X^{i}_s, n(Y_s^i)\right\rangle d|\tilde{K}^i|_s\\
&\leq 2\int_0^t  |X^{i}_s-Y^{i}_s| \sup_{k=1,\cdots,N}|V[\mu_s^N](X_s^k)-V[\rho_s](Y_s^k)| \,ds,
\end{aligned}$$
due to the convexity of the domain $\mo$. Using Lemmas \ref{lem_gron} and \ref{lem:coup_T}, we find 
\[
\sup_{i=1,\cdots,N}|X_t^i-Y_t^i|\leq C\int_0^t \bigl(\|\rho_s\|_{L^1\cap L^{\infty}}\sup_{i=1,\cdots,N}|X_s^i-Y_s^i|+H_N\bigr) \, ds.
\]
Applying Gronwall's inequality, taking the expectation, and using Holder's inequality lead to 
\bq
\label{est_fin}
\E\lt[ \mathcal{W}_{\infty}(\mu_t^N,\rho_t^N)\rt] \leq \E\lt[ \sup_{i=1,\cdots,N}|X_t^i-Y_t^i| \rt] \leq t\exp\lt(\int_0^t \|\rho_s\|_{L^1 \cap L^\infty}\,ds\rt)\E\lt[(H_N)^{2m}\rt]^\frm.
\eq
Finally, we use the convergence estimate obtained in \cite[Theorem 1]{FG} together with the moment estimate in Remark \ref{rmk:mom} to find that for all $t\in[0,T]$
\[
\mathbb{E}\lt[\mathcal{W}_p(\rho_t^N,\rho_t)\rt] \leq C_m N^{-\frac12 + \frac{1}{2m}}+C \left\{ \begin{array}{ll}
N^{-1/2p} + N^{-(q-p)/qp} & \textrm{if $2p > d$ and $q \neq 2p$}, \\[2mm]
N^{-1/2p}\log(1+N)^{1/p} + N^{-(q-p)/qp} & \textrm{if $2p = d$ and $q \neq 2p$},\\[2mm]
N^{-1/d} + N^{-(q-p)/qp} & \textrm{if $2p < d$ and $q \neq d/(d-p)$}.
  \end{array} \right.
\]
Combining the above inequality and \eqref{est_fin} concludes the proof.

%%%%%%%%%%%%%%%%%%%%%%%%%%%%%%%%%
%
%  \appendix
%
%
%%%%%%%%%%%%%%%%%%%%%%%%%%%%%%%%%%%%%

\appendix

%%%%%%%%%%%%%%%%%%%%%%%%%%%%%%%%%
%
%  \section{Gronwall-type inequalities}
%
%
%%%%%%%%%%%%%%%%%%%%%%%%%%%%%%%%%%%%%

\section{Gronwall-type inequalities}\label{app_gron}
In this appendix, we present several Gronwall-type inequalities which used for the estimates of uniform bound and stability for solutions in the current work.

\begin{lemma}\label{lem_gron} Let $f, g$ be nonnegative scalar functions. 
\begin{itemize}
\item[(i)]If $f$ satisfies
\[
f(t) \leq f_0 \exp\lt(C\int_0^t f(s)\,ds \rt), \quad t \geq 0,
\]
then we have
\[
f(t) \leq \frac{f_0}{1 - Cf_0 t}, \quad t \geq 0,
\]
where $C$ is a positive constant. 
\item[(ii)] If $f$ and $g$ satisfy
\[
f^p(t) \leq f_0^p + Cp\int_0^t \lt(g(s) f^{p-1}(s) + f^p(s)\rt)ds, \quad t \geq 0,
\]
with $p \geq 1$, then we have
\[
f(t) \leq f_0e^t + Ce^t \int_0^t g(s) e^{-s}\,ds, \quad t \geq 0,
\]
where $C$ is a positive constant. 
\end{itemize}
\end{lemma}
\begin{proof}(i) Set 
\[
h(t) := f_0 \exp\lt(C\int_0^t f(s)\,ds \rt),
\]
then we get
\[
h'(t) = Ch(t) f(t) \leq Ch(t)^2 \quad \mbox{with} \quad h_0 = f_0.
\]
This yields
\[
(h^{-1}(t))' \geq -C \quad \mbox{and} \quad h(t) \leq \frac{1}{h_0^{-1} - Ct} = \frac{h_0}{1 - Ch_0t}.
\]
Thus we have
\[
f(t) \leq h(t) \leq \frac{f_0}{1 - Cf_0t}.
\]
(ii) Set
\[
F(t):= f_0^p + Cp\int_0^t \lt(g(s) f^{p-1}(s) + f^p(s)\rt)ds.
\]
Then we find that $F$ satisfies 
\[
F'(t) = Cp\lt(g(t) f^{p-1}(t) + f^p(t) \rt) \leq Cp\lt(g(t) F^{(p-1)/p}(t) + F(t) \rt) \quad \mbox{with} \quad F_0 = f_0^p. 
\]
Dividing both sides of the above inequality by $pF^{(1-p)/p}$ implies 
\[
\lt(F^{1/p}(t)\rt)' \leq Cg(t) + CF^{1/p}(t),
\]
and this gives the following inequality for $f$:
\[
f(t) \leq F^{1/p}(t) \leq F_0^{1/p} e^t + Ce^t \int_0^t g(s) e^{-s}\,ds = f_0 e^t + Ce^t \int_0^t g(s) e^{-s}\,ds.
\]
This completes the proof.
\end{proof}

%%%%%%%%%%%%%%%%%%%%%%%%%%%%%%%%%
%
%  \section{Gronwall-type inequalities}
%
%
%%%%%%%%%%%%%%%%%%%%%%%%%%%%%%%%%%%%%

\section{A representation for solutions to the PDE \eqref{1_pde_R}}\label{app_b}
In this part, we show that any weak solution $\rho \in L^{\infty}(0,T;\mathcal{P}(\mo))$ to \eqref{1_pde_R} can be represented as the family of time marginals of a solution to \eqref{eq:NLSDER}.
\begin{lemma}\label{lem:app_b}
For any $\rho \in L^{\infty}(0,T;\mathcal{P}(\mo))$ solution to \eqref{1_pde_R} with initial condition $\rho_0\in \mathcal{P}(\mo)$, there exists a stochastic basis, and on that basis a random variable $Y_0$ with the law $\rho_0$ and an independent Brownian motion $(B_t)_{t\in[0,T]}$ such that the solution to \eqref{eq:NLSDER} generated by the initial condition $Y_0$ and Brownian motion $(B_t)_{t\in[0,T]}$ has time marginals $\rho_t$.
\end{lemma}
\begin{proof} $\textbf{Step A}$ ({\it Regularization}): Let $(\mu_{\e})_{\e>0}$ be a sequence of mollifier and define $\rho_t^\e=\rho_t*\mu_\e$. Then we can easily find that $\rho_t^\e$ satisfies the following Cauchy problem:
\begin{equation*}%\label{eq:regAgg}
\left\{ \begin{array}{ll}
\pa_t \rho^\e_t + \nabla \cdot \lt(\rho^\e_t\V^\e_t\rt) = \sigma \Delta \rho^\e_t, \quad x \in \mathcal{O}, \quad t > 0, &\\[2mm]
\displaystyle \lt\lal \sigma \nabla \rho^\e- \rho^\e\V^\e, n \rt\ral = 0 \quad \mbox{on} \quad \pa \mathcal{O}, &
\end{array} \right.
\end{equation*}
where 
\[
\V^\e_t:=\frac{(\rho_t V[\rho_t])*\mu_\e}{\rho^\e_t}.
\]
Since the vector fields $\V^\e_t$ are smooth and satisfy $\|\V_t^\e\|_{L^\infty}\leq \|V[\rho_t]\|_{L^\infty}\leq \|\nabla \varphi\|_{L^\infty}$, we can define the solution to the following SDE
\begin{equation}\label{eq:regSDE}
\left\{ \begin{array}{ll}
\displaystyle Y_t^\e = Y_0^\e+\int_0^t \V_s^\e(Y_s^\e)\,ds+\sqrt{2\sigma}B_t-K_t^\e,&  \\[3mm]
\displaystyle  K_t^\e=\int_0^t n (Y_s^\e)\,d| K^\e|_s,  \quad | K^\e|_t=\int_0^t \mb_{\pa\mathcal{O}} (Y^\e_s)\,d| K^\e|_s. & 
\end{array} \right.
\end{equation}
By applying It\^o's formula, similarly as in proof of Proposition \ref{lem:EstLinf}, we can also find the time marginals of solution to \eqref{eq:regSDE} is the solution to 
\begin{equation}\label{eq:regA}
\left\{ \begin{array}{ll}
\pa_t \overline{\rho}^\e_t + \nabla \cdot \lt(\overline{\rho}^\e_t\, \V_t^\e\rt) = \sigma \Delta \overline{\rho}^\e_t, \quad x \in \mathcal{O}, \quad t > 0, &\\[2mm]
\displaystyle \lt\lal \sigma \nabla \overline{\rho}^\e- \overline{\rho}^\e\V_t^\e, n \rt\ral = 0 \quad \mbox{on} \quad \pa \mathcal{O}. &
\end{array} \right.
\end{equation}
It is clear that there exists a unique weak solution \eqref{eq:regA} with the initial condition $\rho_0^\e$ due to the linearity together with the smoothness of the vector fields. Thus we have $\LL(Y_t^\e)=\rho_t^\e$.

\noindent $\textbf{Step B}$ ({\it Tightness}): We first show that the family $\{(Y_t^\e+K_t^\e)_{t \in [0,T]}, \e > 0 \}$ is tight. Define $K(R,A)$ as
\[
K(R,A):=\lt\{(f_t)_{t\in[0,T]} \, : \, |f(0)|\leq A  \quad \mbox{and} \quad  \sup_{0\leq s<t\leq T}  \frac{{|f(t)-f(s)|}}{|t-s|^{1/3}}\leq R  \rt\}.
\]
Note that $K(R,A)$ is compact subset of $C([0,T],\R^d)$ due to Ascoli's Theorem. On the other hand, a straightforward computation gives
\[
\lt|\int_s^t \V_u^{\e}(Y^\e_u)\,du+\sqrt{2\sigma}(B_t-B_s)\rt|\leq \lt(T^{2/3}\|\nabla \varphi\|_{L^\infty}+\sup_{0\leq s<t\leq T}\frac{|B_t-B_s|}{|t-s|^{1/3}}\rt)|t-s|^{1/3}=:U_T|t-s|^{1/3}.
\]		
This implies
\[
\IP\lt((Y_t^\e+K_t^\e)_{t\in[0,T]}\notin K(R,A)\rt)\leq \IP\lt( |Y_0^\e|\geq A\rt)+\IP\lt( U_T\geq R\rt).
\] 
Note that $U_T$ is almost surely bounded since the trajectories of the Brownian motion are almost surely $\frac{1}{3}$-H\"older. This together with the tightness of $\rho_0^\e$ yields that for any $\eta>0$ we can find some $R,A>0$ such that 
\[
\sup_{\e>0}\IP\lt((Y_t^\e+K_t^\e)_{t\in[0,T]}\notin K(R,A)\rt)\leq \eta.
\]
Thus the family of the law of the $(Y_t^\e+K_t^\e)_{t\in[0,T]}$ is tight. On the other hand, it follows from \cite[Theorem 1.1]{Li-Si} that the mapping $(w_t)_{t\in[0,T]}$ to $(x_t)_{t\in[0,T]}$ solution to the following Skorokhod problem:
\[
\left\{ \begin{array}{ll}
x_t+k_t=w_t, &\\[2mm]
\displaystyle k_t=\int_0^t h_sn(x_s)\,d|k|_s, \, |k|_t=\int_{0}^{t}\mb_{\pa \mo}(x_s)\,d|k|_s &
\end{array} \right.,		
\]
is continuous. This concludes that the family of the law of $(Y_.^\e)_{\e>0}$ is also tight. 

\noindent $\textbf{Step C}$ ({\it Passing to the limit}): Since the family of the law of $\{(Y_t^\e,B_t)_{t \in [0,T]}, \e > 0 \}$ is tight, it is possible to find a subsequence $\{(Y_t^{\e_n},B_t)_{t \in [0,T]}, \e > 0 \}$ converging in law. Furthermore, up to changing the probability space, we can assume that this convergence holds almost surely due to Skorokhod representation Theorem. For the sake of shortness of notation, we assume the probability space is unchanged. Then it is sufficient to check 
\[
\E\lt[\lt|\int_0^t \lt(\V^\e(Y_s^\e)-V[\rho_s](Y_s)\rt)ds\rt|\rt]\to 0 \quad \mbox{as} \quad \e \to 0.
\]
For this, inspired by the strategy used in \cite[Theorem 2.6]{Figalli} and \cite[Theorem B.1]{FouHau}, we introduce a sequence of continuous function $V^k_.:[0,T]\times \mo \mapsto \R^d$ converging to $V[\rho_.]$ in $L^{1}([0,T]\times \mo, \rho_s(dy)ds)$ as $k \to \infty$. This is possible since the measure $\rho_s(dy)ds$ is a Radon measure. Then we define $\V^{\e,k}_.$ by
\[
\V^{\e,k}_s=\frac{(\rho_sV^k_s)*\mu_{\e}}{\rho_s^\e}.
\]
Using this newly defined vector fields, we estimate as		  
$$\begin{aligned}
&\int_0^t\E\lt[\lt|(\V^\e(Y_s^\e))-V[\rho_s](Y_s)\rt|\rt]ds\cr
&\quad \leq \int_0^t \int_{\mo}|\V_s^{\e}(y)-\V_s^{\e,k}(y) | \rho_s^{\e}(y)\,dyds +\int_0^t \int_{\mo}|\V_s^{\e,k}(y)-V_s^{k}(y) | \rho_s^{\e}(y)\,dyds\\
&\qquad +\int_0^t \int_{\mo}|V_s^{k}(y_1)-V_s^{k}(y_2) | \,\pi_s^{\e}(dy_1,dy_2)ds +\int_0^t \int_{\mo}|V_s^{k}(y)-V[\rho_s](y) | \rho_s(dy)\,ds\\
&\quad =: I_1+I_2+I_3+I_4,
\end{aligned}$$
Here we can easily find that the terms $I_2$ and $I_4$ can be arbitrarily close to zero due to the definition of the sequence $(V^k_.)_{k\geq 0}$. For the estimate of $I_1$, we get
\[
I_1 =\int_0^t \int_{\mo}\lt|\frac{\lt[(V[\rho_s]-V^k_s)\rho_s\rt]*\mu_\e}{\rho_s^\e} \rt| \rho_s^{\e}\,dyds =\int_0^t \int_{\mo}\lt|\lt[(V[\rho_s]-V^k_s)\rho_s\rt]*\mu_\e \rt|dyds.	
\]  
Then we can again use the choice of the sequence $(V^k_.)_{k\geq 0}$ to have that $I_1$ can be arbitrarily small as $k \to \infty$. Finally, we notice that $I_3$ can be rewritten as
\[
I_3=\int_0^t\E\lt[\lt|(V^k_s(Y_s^\e))-V^k_s(Y_s)\rt|\rt]ds.
\]
Then it is clear that $I_3 \to 0$ as $\e \to 0$ since $V^k_.$ is continuous and $(Y^{\e}_t)_{t\in[0,T]}$ converges to $(Y_t)_{t\in[0,T]}$ $\IP$-a.s. This completes the proof. 
\end{proof}

%%%%%%%%%%%%%%%%%%%%%%%%%%%%%%%%%%%%%%%%%%%%%%%%%%%%%%%%%%%%%%%%%%%%%%%%%%%%%%%%%%%%%%%%%%%%%%%%%%%%%%%%%%%%%%%%%%%%%%%%%%%%%%%%%%%%%%%%%%%%%%
%
%
%    Acknowledgments
%
%
%%%%%%%%%%%%%%%%%%%%%%%%%%%%%%%%%%%%%%%%%%%%%%%%%%%%%%%%%%%%%%%%%%%%%%%%%%%%%%%%%%%%%%%%%%%%%%%%%%%%%%%%%%%%%%%%%%%%%%%%%%%%%%%%%%%%%%%%%%%%%%%%
\section*{Acknowledgments}
\small{YPC was partially supported by EPSRC grant EP/K008404/1, ERC-Starting grant HDSPCONTR ``High-Dimensional Sparse Optimal Control'', and Alexander Humboldt Foundation through the Humboldt Research Fellowship for Postdoctoral Researcher. YPC is also supported by National Research Foundation of Korea(NRF) grant funded by the Korea government(MSIP) (No. 2017R1C1B2012918). The authors warmly thank Professor Maxime Hauray for helpful discussion and valuable comments. The authors also acknowledge the Institut Mittag-Leffler, and particularly Professor Jos\'e A. Carrillo, where this work was partially done.}

%%%%%%%%%%%%%%%%%%%%%%%%%%%%%%%%%%%%%%%%%%%%%%%%%%%%%%%%%%%%%%%%%%%%%%%%%%%%%%%%%%%%%%%%%%%%%%%%%%%%%%%%%%%%%%%%%%%%%%%%%%%%%%%%%%%%%%%%%%%%%%
%
%
%         Section: References
%
%
%%%%%%%%%%%%%%%%%%%%%%%%%%%%%%%%%%%%%%%%%%%%%%%%%%%%%%%%%%%%%%%%%%%%%%%%%%%%%%%%%%%%%%%%%%%%%%%%%%%%%%%%%%%%%%%%%%%%%%%%%%%%%%%%%%%%%%%%%%%%%%%%

\end{document}